\documentclass{amsart}
\usepackage{amsfonts}
\usepackage{graphicx}
\usepackage{amscd}

\newtheorem{theorem}{Theorem}
\theoremstyle{plain}

\newtheorem{corollary}{Corollary}

\newtheorem{proposition}{Proposition}
\newtheorem{remark}{Remark}

\numberwithin{equation}{section}

\graphicspath{    {converted_graphics/}    {/}}

\begin{document}
\title{Statistical and Deterministic Dynamics of Maps with Memory}
\thanks{The research of the authors was supported by NSERC grants. The research of Z. Li is also supported
by NNSF of China (No. 11161020 and No. 11361023)}
\subjclass[2000]{37A05, 37A10, 37E05,	37E30}
\date{\today }
\keywords{multivalued maps, selections of multivalued maps, random maps,
absolutely continuous invariant measures}

\author[P. G\'ora]{Pawe\l\ G\'ora }
\address[P. G\'ora]{Department of Mathematics and Statistics, Concordia
University, 1455 de Maisonneuve Blvd. West, Montreal, Quebec H3G 1M8, Canada}
\email[P. G\'ora]{pawel.gora@concordia.ca}
\author[A. Boyarsky]{Abraham Boyarsky }
\address[A. Boyarsky]{Department of Mathematics and Statistics, Concordia
University, 1455 de Maisonneuve Blvd. West, Montreal, Quebec H3G 1M8, Canada}
\email[A. Boyarsky]{abraham.boyarsky@concordia.ca}
\author[Z. Li]{Zhenyang Li }
\address[Z. Li]{Department of Mathematics, Honghe University, Mengzi, Yunnan 661100, China}
\email[Z. Li]{zhenyangemail@gmail.com}

\author[H. Proppe]{Harald Proppe}
\address[H. Proppe]{Department of Mathematics and Statistics, Concordia University,
1455 de Maisonneuve Blvd. West, Montreal, Quebec H3G 1M8, Canada}
\email[H. Proppe]{hal.proppe@concordia.ca}

\begin{abstract}  
We consider  a dynamical  system to have memory if it remembers the current
state as well as the state before that.  The dynamics is defined as follows: $x_{n+1}=T_{\alpha
}(x_{n-1},x_{n})=\tau (\alpha \cdot x_{n}+(1-\alpha )\cdot x_{n-1}),$ where $\tau$ is a one-dimensional map 
on $I=[0,1]$ and  
$0<\alpha <1$ determines how much memory is being used.  $T_{\alpha
}$ does not define a dynamical system since it maps $U=I\times I$ into $I$.  In this note we
let $\tau $ to be the symmetric tent map. We shall prove that
for $0<\alpha <0.46,$ the orbits of $\{x_{n}\}$ are described statistically by
an absolutely continuous invariant measure (acim) in two dimensions. As $\alpha $ approaches $0.5
$ from below, that is, as we approach a balance between the memory state and
the present state, the support of the acims become thinner until at $\alpha
=0.5$, all points have period 3 or eventually possess period 3. For $%
0.5<\alpha <0.75$, we have a global attractor: for all starting points in $U$
except $(0,0)$, the orbits are attracted to the fixed point $(2/3,2/3).$ At $%
\alpha=0.75,$ we have slightly more complicated periodic behavior. 
\end{abstract}

\keywords{Dynamical systems, memory, absolutely continuous invariant measure, global stability}

\maketitle

{Department of Mathematics and Statistics, Concordia University, 1455 de
Maisonneuve Blvd. West, Montreal, Quebec H3G 1M8, Canada}

and

Department of Mathematics, Honghe University, Mengzi, Yunnan 661100, China

\smallskip

E-mails: {abraham.boyarsky@concordia.ca}, {pawel.gora@concordia.ca}, {%
zhenyangemail@gmail.com}, {hal.proppe@concordia.ca}.

\section{ Introduction}

In nonlinear discrete chaotic dynamical systems theory we study the
statistical long term dynamics of iterated maps which depend only on the present
state of the system. In this paper we consider dynamical systems which depend
both on the present state as well as on one previous state. Such memory systems
find applications in cellular automata and in modeling natural phenomena
\cite{Guo,May}. 

Let $\tau $ be a piecewise, expanding map on $I$. We refer to it as  the base map.   At
each step, the system remembers the current state $x_n$
 as well as  one previous state $x_{n-1}$, which we refer to as the memory.
 Our dynamical system is defined by $x_{n+1}=T_{\alpha }(x_{n})=\tau (\alpha \cdot x_{n}+(1-\alpha
)\cdot x_{n-1})$, where $0<\alpha <1$ is a fixed number that specifies the ratio 
between the present state and the memory state. $T_{\alpha }$ does
not define a dynamical system, since it is not a  map of a space into itself.
Rather, it denotes a process. To start a trajectory we need an initial point 
$x_{0}$ and its memory, which we consider to be bundled into the previous
state $x_{n-1}$. When $\alpha $ is close to $0,$ the present state $x_{n}$ is
weighted down, acting as a perturbation on the memory state $x_{n-1}$ which
is dominant. However, when $\alpha $ is close to $1,$ the memory state is
diminished and the resulting system behaves almost like a regular dynamical system, 
depending mostly on the present state. 

In order to define an invariant measure for $T_{\alpha },$ we consider the
2-dimensional map: 
\begin{equation*}
G_{\alpha }:[x_{n-1},x_{n}]\mapsto \lbrack x_{n},T_{\alpha
}(x_{n})]=[x_{n},\tau (\alpha \cdot x_{n}+(1-\alpha )\cdot x_{n-1})]\, ,
\end{equation*}%
i.e., 
\begin{equation*}
G_{\alpha }(x,y)=[y,\tau (\alpha \cdot y+(1-\alpha )\cdot x)]\,.
\end{equation*}%
The trajectory of $G_{\alpha }$ is: 
\begin{equation*}
\lbrack
x_{-1},x_{0}],[x_{0},x_{1}],[x_{1},x_{2}],[x_{2},x_{3}],[x_{3},x_{4}],\dots 
\end{equation*}%
If $\Pi _{1}$ is the projection on the first coordinate, we have 
\begin{equation*}
T_{\alpha ,x_{-1}}^{n}(x_{0})=\Pi _{1}(G_{\alpha }^{n+1}(x_{-1},x_{0}))\ ,\
n=1,2,\dots \,,
\end{equation*}%
where $T_{\alpha ,x_{-1}}$ means that the process $T_{\alpha }$ uses the
initial history, $x_{-1}$.

Let us assume that $G_{\alpha }$ has an ergodic invariant measure $\nu
_{\alpha }$ on $\mathfrak{B}([0,1]^{2})$. The measure $\nu _{\alpha }$
defines a marginal measure $\mu _{\alpha }$ on the first coordinate: $\mu
_{\alpha }(A)=\nu _{\alpha }(A\times \lbrack 0,1])$. In particular, if $\nu
_{\alpha }=g_{\alpha }(x,y)dxdy$, i.e., an absolutely continuous measure with
density $g_{\alpha }(x,y)$, then

\begin{equation*}
\mu _{\alpha }=\left( \int_{[0,1]}g_\alpha(x,y)dy\right) dx
\end{equation*}
is  also absolutely continuous with density $g_{\alpha
,1}(x)=\int_{[0,1]}g_{\alpha }(x,y)dy$.

Since we assume that $G_{\alpha }$ is $\nu _{\alpha }$-ergodic, the
Birkhoff Ergodic Theorem holds. Thus, for any integrable function $f$ and
almost every pair $(x,y),$ we have 
\begin{equation*}
\lim_{n\rightarrow \infty }\frac{1}{n}\sum_{i=0}^{n-1}f(G_{\alpha
}^{i}(x,y))=\int f(x,y)d\nu _{\alpha }(x,y)\,.
\end{equation*}%
If the function $f$ depends only on the first coordinate, we can rewrite
this as 
\begin{equation*}
\lim_{n\rightarrow \infty }\frac{1}{n}\sum_{i=0}^{n-1}f(\Pi _{1}(G_{\alpha
}^{i}(x,y)))=\int f(x)d\mu _{\alpha }(x)\,,
\end{equation*}
that is, 
\begin{equation*}
\lim_{n\rightarrow \infty }\frac{1}{n}\sum_{i=0}^{n-1}f(T_{\alpha
}^{i}(x_{0},y_{0}))=\int f(x)d\mu _{\alpha }(x)\,.
\end{equation*}%
Since the limit is independent of the initial condition, the initial history $x_{-1}$ used by $%
T_{\alpha }$ is irrelevant.

This shows that the marginal measure  of the $G_{\alpha }$-invariant measure
determines the behavior of ergodic averages of trajectories of the process $%
T_{\alpha }$. Thus, $\mu _{\alpha }$ is a good candidate for an ``invariant"
measure of $T_{\alpha }$.

In Section \ref{sec:Prelim}, we show that for certain $\alpha ,$ $G_{\alpha }$ is expanding
in both directions and establish the existence of an $acim$ for the memory
system defined by any piecewise expanding map $\tau .$ In Sections 3 -- 6 we study
the behavior of the memory system defined when the base map is  the tent map $\tau .$ 
For $0<\alpha <0.46,$ we prove the orbits of $\{x_{n}\}$ are described
statistically by an acim. As $%
\alpha $ approaches $0.5$ from below, that is, as we approach a balance
between the memory state and the present state, the support of the acims
become thinner until at $\alpha =0.5$, all points have period 3 or eventually
possess period 3. In Section \ref{sec:deterministic}, we consider $1/2\le \alpha \le 3/4$. 
We prove that for $\alpha=1/2$ all points (except two fixed points) are eventually 
periodic with period 3.
For $\alpha=3/4$ we prove that all points of the line $x+y=4/3$ (except the fixed point)
 are of period 2 and all other points (except $(0,0)$) are attracted to this line. 
For $1/2< \alpha < 3/4$, we prove the
existence of a global attractor: for all starting points in the square $[0,1\times[0,1]$ except $(0,0)$,
the orbits are attracted to the fixed point $(2/3,2/3).$ 

Additional pictures illustrating the behaviour of the family $G_\alpha$ and Maple programs used in this study can be found
at

 http://www.mathstat.concordia.ca/faculty/pgora/G-map/.

\section{ Preliminary Results }\label{sec:Prelim}

In this section we  show that for certain $\alpha ,$ $G_{\alpha }$ is expanding in both
directions.(We will ussually suppres the subscript $\alpha$ in the sequel.)

Let $\tau :I\rightarrow I$ be a piecewise expanding map is defined on the
partition $\mathcal{P}$ with endpoints $\{a_{0}=0,a_{1},a_{2},\dots
,a_{q-1},a_{q}=1\}$. Let $I_{i}=[a_{i-1},a_{i}]$, $i=1,2,\dots ,q$. Then,
the map $G_{\alpha }$ is defined on the partition whose boundaries are the
boundaries of the square $U=I^{2}$ and the lines 
\begin{equation*}
L_{i}^{\alpha }:y=\frac{a_{i}}{\alpha }-\frac{1-\alpha }{\alpha }x\ ,\
i=0,2,\dots ,q\,.
\end{equation*}%
Each of the lines $L_{0}^{\alpha }$ and $L_{q}^{\alpha }$ intersects $%
[0,1]^{2}$ at only one point. Let $R_{i}$ denote the region in $[0,1]$
between the lines $L_{i-1}^{\alpha }$ and $L_{i}^{\alpha }$, $i=1,2,\dots ,q$%
. The example for $\mathcal{P}=\{0,0.25,0.5,0.8,1\}$ is shown in Figure \ref{fig:partitionG} a
and the example for $\mathcal{P}=\{0,0.5,1\}$ is shown in Figure \ref{fig:partitionG} b.

\begin{figure}[h] 
  \centering
 \includegraphics[bb=0 0 900 400,width=5.67in,height=2.52in,keepaspectratio]{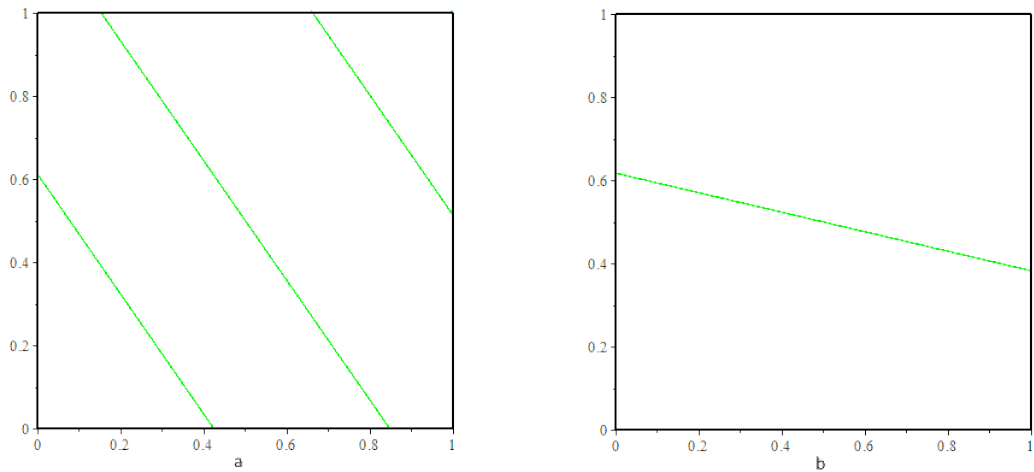}
 \caption{Examples of partitions for map $G$}
  \label{fig:partitionG}
\end{figure}

Note that $G_{\alpha }$ is not piecewise expanding. However, we will show
that $G_{\alpha }^{2}$ is a piecewise expanding map for small values of $%
\alpha $. The inverse branches of $G_{\alpha }^{2}$ are of the form $%
(G_{\alpha ,j}\circ G_{\alpha ,k})^{-1}=G_{\alpha ,k}^{-1}\circ G_{\alpha ,j}^{-1}$%
. We have $D(G_{\alpha ,k}\circ G_{\alpha ,j})^{-1}=DG_{\alpha ,j}^{-1}\circ
G_{\alpha ,k}^{-1}\cdot DG_{\alpha ,k}^{-1}$. That is 
\begin{equation}
DG_{\alpha ,j}^{-1}\circ G_{\alpha ,k}^{-1}\cdot DG_{\alpha ,k}^{-1}=%
\begin{pmatrix}
\frac{-\alpha }{1-\alpha } & \frac{1}{(1-\alpha )\tau _{j}^{\prime }(\tau
_{j}^{-1}(u))} \\ 
1 & 0 \\ 
\end{pmatrix}%
\cdot 
\begin{pmatrix}
\frac{-\alpha }{1-\alpha } & \frac{1}{(1-\alpha )\tau _{k}^{\prime }(\tau
_{k}^{-1}(v))} \\ 
1 & 0 \\ 
\end{pmatrix}%
,  \label{eq:DGkGj_inverse}
\end{equation}%
which is equal to 
\begin{equation}
\begin{pmatrix}
\frac{\alpha ^{2}}{(1-\alpha )^{2}}+\frac{1}{(1-\alpha )\tau _{j}^{\prime
}(\tau _{j}^{-1}(u))} & \frac{-\alpha }{(1-\alpha )^{2}\tau _{k}^{\prime
}(\tau _{k}^{-1}(v))} \\ 
\frac{-\alpha }{1-\alpha } & \frac{1}{(1-\alpha )\tau _{k}^{\prime }(\tau
_{k}^{-1}(v))} \\ 
\end{pmatrix}%
.
\end{equation}%
If $\alpha $ is chosen small enough, since $\tau $ is expanding, all the
entries of the matrix can be made smaller than one (in absolute value), so
the norm is smaller than one. This implies that $G_{\alpha }^{2}$ is a
piecewise expanding map.  By \cite{Tsu} we have the existence of an acim.

One can immediately make the following observation.

\begin{remark}
If $\alpha \approx 0$ (strong memory), then $G_{\alpha }(x,y)\approx (y,\tau
(x))$ hence $G_{\alpha }^{2}(x,y)\approx (\tau (x),\tau (y))$, so $G_{\alpha
}$ is likely to have an $acim$ because $\tau $ has an $acim$ and $G_{\alpha }
$ is close to the product $\tau \times \tau $. On the other hand, if $\alpha
\approx 1$ (weak memory), then $G_{\alpha }(x,y)\approx (y,\tau (y))$, which
is independent of $x$ and the orbit of any point $(x,y)\in U$ is
approximately a subset of the graph of $\tau $. In this case it is likely
that there is an SRB measure, but that it is singular with respect to the 2D
Lebesgue measure.
\end{remark}

We now show that in general $G_{\alpha }$ is not piecewise expanding. Suppose $\tau
_{j}:I_{j}=(a_{j},b_{j})\rightarrow I$ is a monotonic branch of $\tau $.
Then $G_{\alpha }$ is piecewise monotonic on the strips $\{(x,y):a_{j}<%
\alpha y+(1-\alpha )x<b_{j}\}$. If $G_{j}$ is the branch of $G_{\alpha }$
corresponding to $I_{j}$, then the inverse of $G_{\alpha ,j}$ is given by 
\begin{equation}
G_{_{\alpha ,j}}^{-1}(u,v)=\left( \frac{\tau _{j}^{-1}(v)-\alpha u}{1-\alpha 
},u\right) 
\end{equation}

Note that 
\begin{equation}
D_{(u,v)}G_{\alpha ,j}^{-1}=%
\begin{pmatrix}
\frac{-\alpha }{1-\alpha } & \frac{1}{(1-\alpha )\tau _{j}^{\prime }(\tau
_{j}^{-1}(v))} \\ 
1 & 0 \\ 
\end{pmatrix}%
.  \label{eq:DGj_inverse}
\end{equation}

Such a matrix has Euclidean norm $\left\Vert {DG_{\alpha ,j}^{-1}}%
\right\Vert _{2}\geq 1$. Indeed, for a square matrix $A$, this norm is equal
to $\sqrt{\lambda _{\max }(A^{T}A)}$, where $\lambda _{\max }(A^{T}A)$
denotes the maximum eigenvalue of the symmetric matrix $A^{T}A$. For us, $A$
is given by \eqref{eq:DGj_inverse} and is of the form 
\begin{equation}
\begin{pmatrix}
a & b \\ 
1 & 0 \\  
\end{pmatrix}%
.  \label{eq:A}
\end{equation}%
Therefore, $A^{T}A$ is of the form 
\begin{equation}
\begin{pmatrix}
1+a^{2} & ab \\ 
ab & b^{2} \\ 
\end{pmatrix}%
.  \label{eq:ATA}
\end{equation}%
Note that the sum of the eigenvalues of a matrix is equal to the trace of
the matrix, which for $A^{T}A$ is $1+a^{2}+b^{2}$. This means that both
eigenvalues cannot be smaller than $1$. Therefore, $\lambda _{\max
}(A^{T}A)\geq 1$, and $G$ is not a piecewise expanding map (see 
\cite{Sau00}, Remark 2.1 item 2) in the sense that all directions are
contracted under the branches of the inverse of $G$.

\section{  $\tau$ is the symmetric tent map.}\label{sec:special}

In the sequel we study the dynamical system where the base map is
 $$\tau(x)=\begin{cases} 2x &\ , \ \text{for }\ 0\le x< 1/2 ;\\
                                   2-2x &\ , \ \text{for }\ 1/2\le x\le 1.
\end{cases}$$
and $$G_\alpha(x,y)=(y,\tau(\alpha y+(1-\alpha)x)).$$

\begin{figure}[h] 
  \centering
  \includegraphics[bb=0 -1 662 339,width=3.92in,height=2.01in,keepaspectratio]{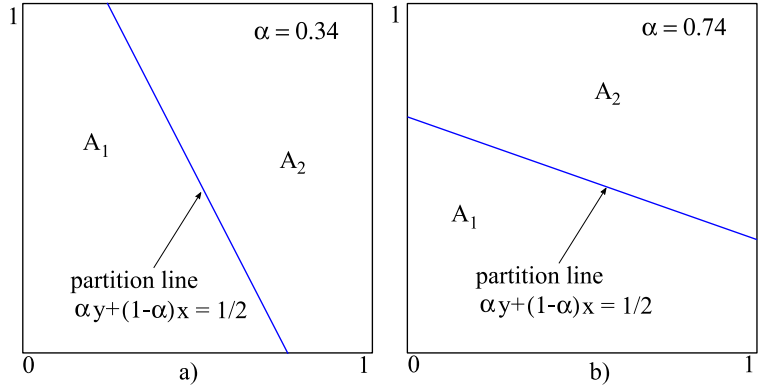}
  \caption{Partition into $A_1$ and $A_2$ for a) $\alpha=0.34$ and  b) $\alpha=0.74$}
  \label{fig:partitions034and074}
\end{figure}

\section{Case I: $0\le\alpha<1/2$.}

\begin{remark} For $\alpha=0$ we have $$(x,y)\overset {G}{\longrightarrow}(y,\tau(x))\overset {G}{\longrightarrow}(\tau(x),\tau(y)),$$
so $G^2=\tau\times\tau$ and preserves two-dimensional Lebesgue measure on the square $[0,1]\times [0,1]$.
\end{remark}

In the sequel we consider only $\alpha>0$.

 Let $A_1$ denote the part of the square $[0,1]\times [0,1]$ below the line $\alpha y+(1-\alpha)x=1/2$ and $A_2$ the part above this line.  We now collect some simple  facts.

\begin{proposition}\label{pr1} If $(x,y)\in A_1$  and $\alpha y+(1-\alpha)x> a $, $a<1/2$, then the point $(w,z)=G(x,y)$ satisfies $\alpha z+(1-\alpha)w> 2\alpha a$. 
\end{proposition}

\begin{proof} We have $$\alpha z+(1-\alpha)w= \alpha(2\alpha y+2(1-\alpha)x)+(1-\alpha)y= 
[2\alpha^2-\alpha+1]y+2\alpha(1-\alpha)x\, .$$
It is enough to see that $(2\alpha^2-\alpha+1)/\alpha=2\alpha-1+1/\alpha>
2\alpha$ and $2\alpha(1-\alpha)/(1-\alpha)=2\alpha$.
\end{proof}

\begin{proposition} \label{pr2} If $(x,y)\in A_1$ and $G(x,y)\in A_1$ as well, and $\alpha y+(1-\alpha)x> a $, $a<1/2$, then the point $(w,z)=G^2(x,y)$ satisfies $\alpha z+(1-\alpha)w> (4\alpha^2-2\alpha+2) a$. 
\end{proposition}

\begin{proof} We have $$(w,z)=\left( 2(1-\alpha)x+2\alpha y, 4\alpha(1-\alpha)x+(4\alpha^2-2\alpha+2)y\right)\, ,$$
and
$$\begin{gathered}\alpha z+(1-\alpha)w= (-4\alpha^3+6\alpha^2-4\alpha+2)x+(4\alpha^3-4\alpha^2+4\alpha)y\\ 
=(4\alpha^2-2\alpha+2)(1-\alpha)x+(4\alpha^2-4\alpha+4)\alpha y\, .\end{gathered}$$
It is enough to see that $(4\alpha^2-4\alpha+4)>(4\alpha^2-2\alpha+2)>1$.
\end{proof}

\begin{proposition}\label{pr3} If $(x,y)\in A_2$, then the point $(w,z)=G(x,y)$ satisfies $\alpha z+(1-\alpha)w\ge 2\alpha^2 $. 
\end{proposition}

\begin{proof} We have $$\alpha z+(1-\alpha)w= \alpha(2-2\alpha y-2(1-\alpha)x)+(1-\alpha)y= 2\alpha+[1-2\alpha^2-\alpha]y-2\alpha(1-\alpha)x\, .$$
 For $\alpha\in(0,1/2)$ the coefficient next to $y$ is positive and that next to $x$ is negative so the minimum is reached at  $(1,0)$ and is equal to $2\alpha^2$.
 This completes the proof.
\end{proof}

\begin{proposition}\label{pr4} If $(x,y)\in A_2$ and $G(x,y)\in A_1$  then the point $(w,z)=G^2(x,y)$ satisfies $\alpha z+(1-\alpha)w\ge 2\alpha(1-\alpha)\ge 2\alpha^2 $. 
\end{proposition}

\begin{proof} We have $$\alpha z+(1-\alpha)w=
 -(4\alpha^2-2\alpha+2)(1-\alpha)x-4\alpha^3y+4\alpha^2-2\alpha+2\, .$$
 For $\alpha\in(0,1/2)$ both coefficients next to $x$ and $y$ are negative so the minimum is reached at  $(1,1)$ and is equal to $2\alpha(1-\alpha)\ge 2\alpha^2$.
 This completes the proof.
\end{proof}

\begin{proposition}\label{pr5} Let $A_I$ denote the part of  the square $[0,1]\times [0,1]$ above (to the right of) the line $\alpha y+(1-\alpha)x=2\alpha^2$. 
 Propositions \ref{pr1}-\ref{pr3} prove that the support of $G$-invariant measures (except the point measure at $(0,0)$) must lie in region $A_I$.
\end{proposition}

\begin{proof} Proposition \ref{pr2} implies that every  point of $A_1$, except $(0,0)$, enters $A_2$ after a finite number of steps.
Let us consider a point $X_0\in A_2$. By Proposition \ref{pr3} its image $X_1=G(X_0)$  stays above the line $\alpha y+(1-\alpha)x=2\alpha^2$.
Assuming that $X_1\in A_1$, by Proposition \ref{pr4} the point $X_2=G(X_1)$ is  also above this line. If $X_2\in A_1$ the next image 
$X_3=G(X_2)=G^2(X_1)$ is above the line $\alpha y+(1-\alpha)x=2\alpha^2(4\alpha^2-2\alpha+2)$ (by Proposition \ref{pr2}). Now, if $X_3\in A_1$, 
the next image $X_4=G(X_3)=G^2(X_2)$ is also above this line.
We see that further points of the  trajectory
move up towards $A_2$ and none of them  can go below  the line $\alpha y+(1-\alpha)x=2\alpha^2$.
\end{proof}

\begin{remark} For $0.24<\alpha<1/2$, if $(x,y)\in A_I$, then   it reaches $A_2$ in at most 6 steps.\end{remark}


We define some functions which we will use below. Let $G_i=G_{|A_i}$, $i=1,2$, be the restrictions of $G$ to regions $A_1$ and $A_2$, respectively.
Let $S(x,y)=\alpha y+(1-\alpha)x$. Then, $A_1=\{ (x,y): 0\le x,y\le 1, S(x,y)<1/2\}$ and $A_2=\{ (x,y): 0\le x,y\le 1, S(x,y)\ge 1/2\}$.

Let $$D_1=DG_1=\left[\begin{matrix} 0&1\\
                            2(1-\alpha)&2\alpha  \end{matrix}\right]  \ \ \  ,\  \ \ 
D_2=DG_1=\left[\begin{matrix} 0&1\\
                            -2(1-\alpha)&-2\alpha  \end{matrix}\right] .$$

\begin{theorem} The map $G$ admits an  acim for $0<\alpha\le \alpha_1\sim 0.24760367$
\end{theorem}

We define  $\alpha_1$ as a root of the equation 
$16\alpha^4+16\alpha^3-52\alpha^2+48\alpha-9=0$ in the interval $[0,1]$. It is explained below.

\begin{proof}
We will prove that $G^2$ satisfies the assumptions of Tsujii (\cite{Tsu}), i.e., it is piecewise analytic and
 expanding in the sense that for any vector $v$ we have $|DG^2v|>|v|$.
We will do this by showing that the  smaller singular value $s_2(\alpha)$ of the matrix $D_iD_j$, $i,j\in \{1,2\}$ is above 1  for $0<\alpha\le 0.24760367$.

\begin{figure}[h] 
  \centering
  \includegraphics[bb=0 -1 651 315,width=3.92in,height=1.9in,keepaspectratio]{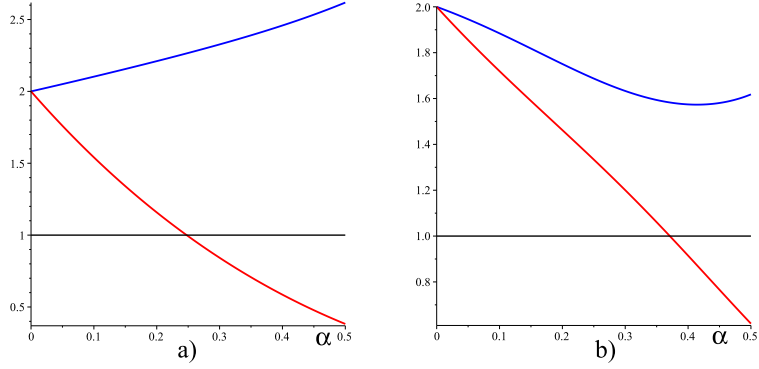}
  \caption{a) Singular values for matrices $D_2D_1$ and $D_1D_1$. The lower curve intersects level 1 at $\alpha_1\sim 0.24760367$.
b) Singular values for matrices $D_2D_2$ and $D_1D_2$. The lower curve intersects level 1 at $\sim 0.3709557543$.}
  \label{fig:singularvaluesD2D1andD2D2}
\end{figure}

The singular values of the matrices $D_2D_1$ and $D_1D_1$ are
$$\sigma_1(\alpha)=\sqrt{16 \alpha^4-24 \alpha^3+22 \alpha^2-8 \alpha+4+2 \sqrt{w_1(\alpha)}},$$
and
$$\sigma_2(\alpha)=\sqrt{16 \alpha^4-24 \alpha^3+22 \alpha^2-8 \alpha+4-2 \sqrt{w_1(\alpha)}},$$
where
$$w_1(\alpha)=64 \alpha^8-192 \alpha^7+320 \alpha^6-328 \alpha^5+245 \alpha^4-120 \alpha^3+36 \alpha^2.$$

They are shown in Figure \ref{fig:singularvaluesD2D1andD2D2} a). The lower curve intersects level 1 at at the root $\alpha_1$ of $16\alpha^4+16\alpha^3-52\alpha^2+48\alpha-9=0$, i.e., at $\alpha_1\sim 0.24760367$.

The singular values of the matrices $D_2D_2$ and $D_1D_2$ are:
$$\sigma_1(\alpha)=\sqrt{16 \alpha^4-8 \alpha^3+6 \alpha^2-8 \alpha+4+2 \sqrt{w_2(\alpha)}},$$
and
$$\sigma_2(\alpha)=\sqrt{16 \alpha^4-8 \alpha^3+6 \alpha^2-8 \alpha+4-2 \sqrt{w_2(\alpha)}},$$
where
$$w_2(\alpha)=64 \alpha^8-64 \alpha^7+64 \alpha^6-88 \alpha^5+69 \alpha^4-24 \alpha^3+4 \alpha^2 .$$

They are shown in Figure \ref{fig:singularvaluesD2D1andD2D2} b). The lower curve intersects level 1 at $\sim 0.3709557543$.
This shows that at least for $0<\alpha\le\alpha_1\sim 0.24760367$ the assumptions of \cite{Tsu} are satisfied and thus, $G^2$ and consequently also $G$
admit an acim. 
\end{proof}

\section{Proof of the existence of acim for $ \alpha>\alpha_1$}\label{sec:above 0.24}

We will prove that high iterates of the map $G$ expand all  vectors. We will make estimates of  the smaller singular value $\sigma_2$ of 
derivative matrix $DG^n$ for large $n$. The general strategy is as follows: we will consider the admissible products of the derivative matrices
$\prod_{j=1}^{n_s}D_{i_j}$, where $i_j\in \{1,2\}$ and the length $n_s$ depends on the sequence, for $\alpha\in (\alpha_s,\alpha_t)$,
 where $(\alpha_s,\alpha_t)$ denotes contiguous intervals. The order of the matrices is natural, e.g.,  the sequence $D_1D_2D_2$ corresponds 
to the iteration $G_1G_2G_2$. We will consider sequences of the form $D_1^nD_2^m$, $n\le 3$, $m\ge 1$, since by Proposition \ref{pr2} every point (except $(0,0)$) visits region $A_2$. 
We will break the long sequence into short ``good" sequences for which we can bound $\sigma_2$ from below by  numbers larger that 1. Since 
\begin{equation}\label{ineqs2} \sigma_2(AB)\ge \sigma_2(A)\sigma_2(B), \end{equation}
this will allow us to show that the $\sigma_2$ of a long product grows to infinity with the length $n_s$.
 Once we have a good estimate, we proceed as follows: we choose a large number $M$ and find a sequence length $n_s$
such that any admissible sequence of length $n_s$ starting with $D_2$ has $\sigma_2>M$. Then, adding at most
 three matrices $D_1$ at the beginning of the sequences and a corresponding number of matrices at the end
(to keep the length of all sequences equal to $n_s+3$) we will have derivative matrices of $G^{n_s+3}$ 
for all non-transient points (we will prove that 3 is enough) and their $\sigma_2$'s  greater than 1. This
proves that $G^{n_s+3}$ on the set of non-transient points expands all vectors and in turn that $G$ admits an acim.

Our proofs are based on symbolic calculations using Maple 17, but they are all finite calculations and ``in principle" could be done using pen and paper.

Recall  $G_i=G_{|A_i}$, $i=1,2$ are the restrictions of $G$ to regions $A_1$ and $A_2$, respectively.

\begin{figure}[h] 
  \centering
  \includegraphics[bb=0 -1 297 300,width=2.73in,height=2.77in,keepaspectratio]{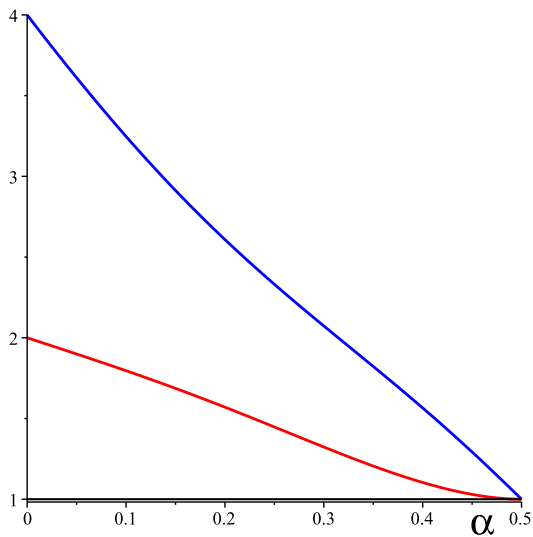}
  \caption{Singular values of  $D_1D_2D_2$ or $D_2D_2D_2$.}
  \label{fig:singularD2D2D2}
\end{figure}

The following result holds for all $0< \alpha<1/2$.

\begin{proposition}\label{goodseq} For any matrix $M$ we have $\sigma_2(D_1M)=\sigma_2(D_2M)$. Also, 
\begin{equation}\label{D1D2D2} \sigma_2(D_1D_2D_2)=\sigma_2(D_2D_2D_2)>1,
\end{equation}
for $0<\alpha<1/2$. More generally,
$$\sigma_2(D_1D_2^m)\ge\sigma_2(D_2D_2D_2)>1,\text{ for }  m=2+3k , k \ge 1 , 0< \alpha<1/2.$$

\end{proposition}
\begin{proof} The singular values of the matrix $B$ are square roots of the eigenvalues of the matrix $B^TB$, where $B^T$ is the transpose of $B$. 
Since $D_1^TD_1=D_2^TD_2$, the first claim follows. The graphs of the singular values of the matrices $D_1D_2D_2$ and $D_2D_2D_2$ are shown in Figure \ref{fig:singularD2D2D2}.
Both singular values are above 1 for all $0\le\alpha<1/2$. The last inequality follows from (\ref{ineqs2}).
\end{proof}

\begin{figure}[h] 
  \centering
  \includegraphics[bb=0 0 486 244,width=3.92in,height=1.96in,keepaspectratio]{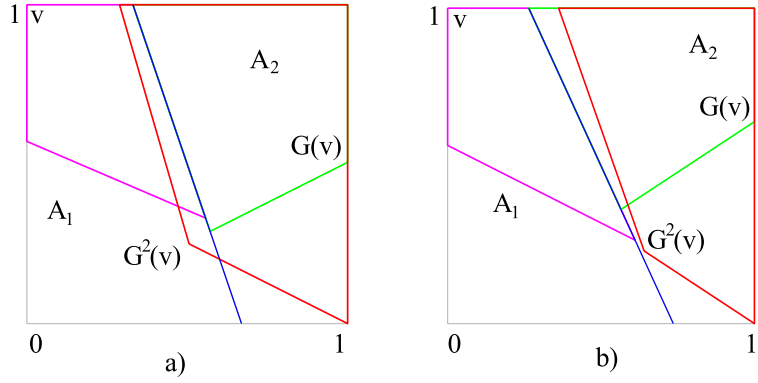}
  \caption{First two images of $A_1$ for  a) $\alpha=0.25290169942$ and  b) $\alpha=0.320169942$}
  \label{fig:atleast2times}
\end{figure}

\begin{proposition}\label{atleast2times} For $\alpha>(\sqrt{5}-1)/4\sim 0.3090169942$ a point in $A_2$ originating in  $A_1$ must stay in $A_2$ for at least 2 steps.
\end{proposition}
\begin{proof} Figure \ref{fig:atleast2times} shows the first (green) and second (red) image of $A_1$. $G^{-1}(A_2)\cap A_1$ is bounded by magenta lines, the blue line is the partition line $S(x,y)=1/2$.
The important point is $v_2=G(G(v))=(2\alpha,2\alpha(1-2\alpha))$ for $v=(0,1)$. When $v_2\in A_1$, then points can return to $A_1$ after one visit in $A_2$. When $v_2\in A_2$, then a point coming from $A_1$ must stay in $A_2$ for at least 2 steps. $S(v_2)=1/2$ for $\alpha=(\sqrt{5}-1)/4\sim 0.3090169942$.
\end{proof}

\begin{proposition}\label{estimates1} The following estimates of $\sigma_2(D_1^nD_2^m)$ for various $n$ and $m$ were obtained using Maple 17: 

1) $\sigma_2(D_1D_2)> 1$  at least for $\alpha\le  0.3709557543$; 

2) $\sigma_2(D_1D_2D_2D_2)> 1$  at least for $\alpha\le  0.3938896523$; 

3) $\sigma_2(D_1D_1D_2)> 1$  at least for $\alpha\le  0.3149466135$;

4) $\sigma_2(D_1D_1D_2D_2)> 1$  at least for $\alpha\le  0.3758203590$;

5) $\sigma_2(D_1D_1D_2D_2D_2)> 1$  at least for $\alpha\le  0.3506831157$;

6) $\sigma_2(D_1D_1D_1D_2)> 1$  at least for $\alpha\le  0.3058009335$;

7) $\sigma_2(D_1D_1D_1D_2D_2)> 1$  at least for $\alpha\le  0.3355882883$;

8) $\sigma_2(D_1D_1D_1D_2D_2D_2)> 1$  at least for $\alpha\le 0.3312697596$;

\end{proposition}


\begin{theorem}\label{24-27} The map $G$ admits an acim for $ \alpha_1\le\alpha\le\alpha_2\sim 0.2797707433 $.
\end{theorem}
We define $\alpha_2$ as  a root  of the equation $8\alpha^4-8\alpha^3+8\alpha^2=1/2$ in $[0,1]$.
Again, it is explained below in Proposition \ref {2times}.

\begin{figure}[h] 
  \centering
  \includegraphics[bb=0 -1 632 335,width=3.92in,height=2.08in,keepaspectratio]{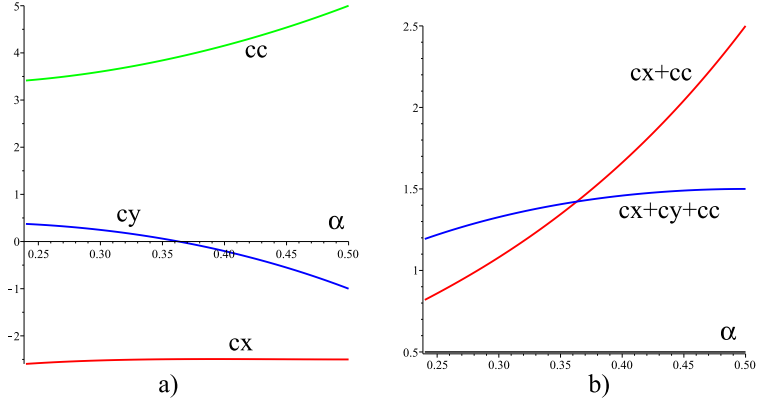}
  \caption{a) Functions $cx,cy,cc$ in Proposition \ref{3times}. b)Functions $cx+cc$ and $cx+cy+cc$ in Proposition \ref{3times}.}
  \label{fig:atmost3times_aib}
\end{figure}

First, we prove the following:
\begin{proposition}\label{3times} For $ \alpha_1\le\alpha\le\alpha_2$, a point originating in $A_2$ remains in 
$A_1$ for at most 3 steps.
\end{proposition}

\begin{proof} It is enough to show that $f(x,y)=S(G_1(G_1(G_1(G_2(x,y)))))\ge 1/2$. We have 
$$f(x,y)=cx(\alpha)x+ cy(\alpha)y+ cc(\alpha), $$
where 
\begin{equation*}\begin{split}cx(\alpha)&=16\alpha^5-40\alpha^4+52\alpha^3-36\alpha^2+12\alpha-4;\\
cy(\alpha)&=-16\alpha^5+16\alpha^4-12\alpha^3-8\alpha^2+4\alpha;\\
cc(\alpha)&=16\alpha^4-24\alpha^3+28\alpha^2-8\alpha+4.
\end{split}
\end{equation*}

The functions $cx$, $cy$ and $cc$ are shown in Figure \ref{fig:atmost3times_aib} a). We consider the worst case scenario, i.e.,
$y=1$ and $x=0$ where $cx>0$ and  $x=1$ where $cx<0$. Graphs of $cx+cc$ and $cx+cy+cc$ are shown in Figure \ref{fig:atmost3times_aib} b).
They both above $1/2$ for $\alpha> 0.24$, and in particular for $ \alpha_1\le\alpha\le\alpha_2$.
\end{proof}

\textbf{Proof of Theorem \ref{24-27}:} By Proposition \ref{3times}, Proposition \ref{goodseq} and estimates
of Proposition \ref{estimates1} we see that, for   $\alpha$'s in the interval $ [\alpha_1\,\alpha_2]$,
all admissible ``basic" sequences of derivative matrices have $\sigma_2$ larger than 1.
Note that we have
 \begin{equation}\label{ineqsigma}
\sigma_2(D_1^nD_2^m\ge \sigma_2(D_1^nD_2^{m-3})\sigma_2(D_2^3)>\sigma_2(D_1^nD_2^{m-3}),
\end{equation}
for $m>3$ (Proposition \ref{goodseq}). This shows that the general strategy described at the beginning of
this section will work and proves the theorem.

\begin{theorem}\label{27-34}The map $G$ admits an acim for $ \alpha_2\le\alpha\le\alpha_3=1/3 $.
\end{theorem}

\begin{figure}[tbp] 
  \centering
  \includegraphics[bb=0 0 447 228,width=3.92in,height=2in,keepaspectratio]{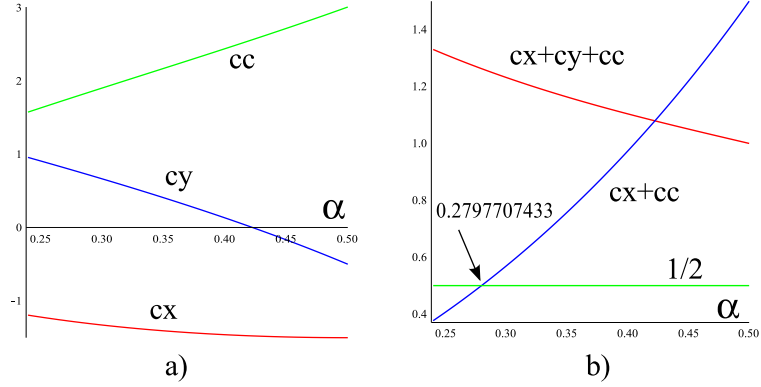}
  \caption{Functions $cx,cy,cc$ and their sums in Proposition \ref{2times}}
  \label{fig:27-34}
\end{figure}

First, we prove the following:
\begin{proposition}\label{2times} For $ \alpha_2\le\alpha\le\alpha_3$ a point coming from $A_2$ can stay in 
 $A_1$ for at most 2 steps.
\end{proposition}

\begin{proof} The proof is similar to that of Proposition \ref{3times}.
It is enough to show that $f(x,y)=S(G_1(G_1(G_2(x,y))))\ge 1/2$. We have 
$$f(x,y)=cx(\alpha)x+ cy(\alpha)y+ cc(\alpha), $$
where 
\begin{equation*}\begin{split}cx(\alpha)&=8\alpha^4-16\alpha^3+16\alpha^2-8\alpha;\\
cy(\alpha)&=-8\alpha^4+4\alpha^3-2\alpha^2-4\alpha+2;\\
cc(\alpha)&=8\alpha^3-8\alpha^2+8\alpha.
\end{split}
\end{equation*}

The functions $cx$, $cy$ and $cc$ are shown in Figure \ref{fig:27-34} a). Again, we consider the worst case scenario, i.e.,
$y=1$ and $x=0$ where $cx>0$ and  $x=1$ where $cx<0$. Graphs of $cx+cc$ and $cx+cy+cc$ are shown in Figure \ref{fig:27-34} b).
They are both above $1/2$ for $\alpha> \alpha_2$.
\end{proof}

\textbf{Proof of Theorem \ref{27-34}:}
Let us first consider the sequence   $D_1D_1D_2$. By part 3) of Proposition \ref{estimates1}  
its $\sigma_2$ is larger than 1 until $\alpha\sim 0.3149466135$. By Proposition \ref{atleast2times}
 the sequence is not admissible after $\alpha\sim 0.3090169942$. All other admissible ``basic" sequences of derivative matrices have 
$\sigma_2$ larger than 1 for   $\alpha$'s in the interval $ [\alpha_2,\alpha_3]$.
We used Proposition \ref{2times}, Proposition \ref{goodseq} and estimates
of Proposition \ref{estimates1} as well as inequality (\ref{ineqsigma}).
  This shows that the general strategy described at the beginning of
this section will work and proves the theorem.

\section{Proof of the existence of acim for $ \alpha>\alpha_3=1/3$}\label{sec:above1by3}

We will continue the estimates of $\sigma_2$ for ``basic" admissible sequences.

\begin{figure}[tbp] 
  \centering
  \includegraphics[bb=0 0 537 289,width=3.92in,height=2.11in,keepaspectratio]{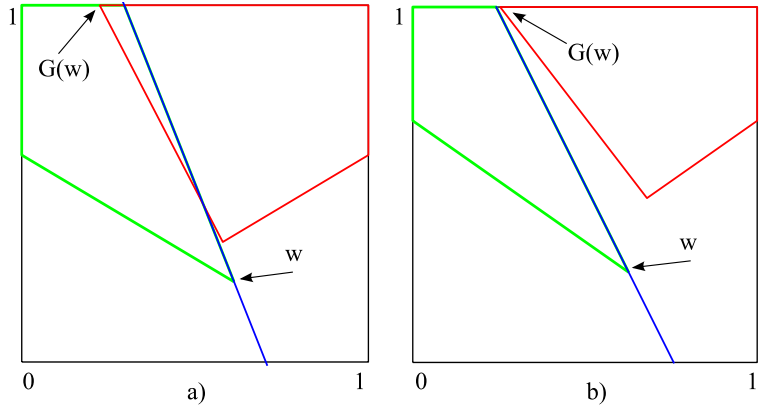}
  \caption{Region $G(A_2)\cap A_1$ and its image for a) $\alpha=0.29$ and b) $\alpha=0.34$}
  \label{fig:34-1}
\end{figure}

\begin{proposition}\label{1times} For $ \alpha>\alpha_3=1/3$ a point coming from $A_2$ can stay in $A_1$ for at most 1 step.
\end{proposition}

\begin{proof} Figure \ref{fig:34-1} shows the region $G(A_2)\cap A_1$ (outlined in green) and its image (outlined in red). The blue line is the partition line $S(x,y)=1/2$.
The important point is $G(w)=(\alpha/(\alpha+1),1)$ for $w=\left(\frac{\alpha+1/2}{\alpha+1},
\frac \alpha{\alpha+1}\right)$. 
When $G(w)\in A_1$,  points coming from $A_2$ can stay in $A_1$ for two steps. When $G(w)\in A_2$, a point coming from $A_2$ can be in $A_1$ for only one step. $S(G(w))=2\alpha/(\alpha+1)$ so $S(G(w))=1/2$
for $\alpha_3=1/3$.
\end{proof}

Proposition \ref{atleast2times} and Proposition \ref{1times} imply that for $\alpha>1/3$ basic admissible
sequences are of the form $D_1D_2^m$ with $m\ge 2$.

\begin{proposition}\label{estimates2} For $ \alpha>\alpha_3=1/3$ we give estimates  of $\sigma_2$ for
basic admissible  sequences. Again, the estimates are obtained using Maple 17.

1) $ \sigma_2(D_1D_2^3)>1$  at least for $\alpha\le  0.3938896523$;

2) $\sigma_2(D_1D_2^4)>1$  at least for $\alpha\le  0.4444154417$;

3) $\sigma_2(D_1D_2^6)>1$  at least for $\alpha\le  0.4345268819$;

4) $\sigma_2(D_1D_2^7)>1$  at least for $\alpha\le  0.4645618403$;

\end{proposition}

\begin{corollary} Propositions \ref{atleast2times}, \ref{1times} and \ref{estimates2} imply that
our general strategy works for $\alpha\le 0.3938896523$. For longer sequences we use inequality
(\ref{ineqsigma}).
\end{corollary}

\begin{figure}[h] 
  \centering
  \includegraphics[bb=0 0 320 333,width=3.92in,height=4.07in,keepaspectratio]{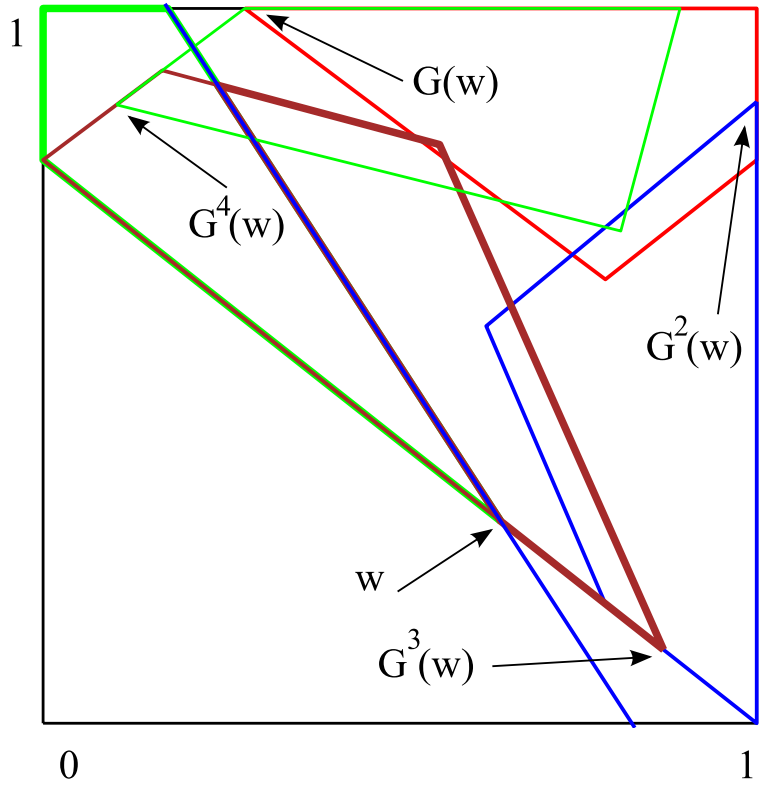}
  \caption{Four first images of $G(A_2)\cap A_1$, $\alpha> 0.39$}
  \label{fig:2after3D2_1}
\end{figure}

\begin{proposition}\label{0.416} For $\alpha> 0.3931078326$, the sequence $D_1D_2D_2D_2$ is followed by $D_1D_2D_2$.
We have 
$$\sigma_2(D_1D_2^2D_1D_2^3)>1 \ \text{ for at least  }\ \alpha\le  0.4160029431.$$
With the previous results this extends the interval of the existence of acim up to $\alpha\sim 0.4160029431$.
\end{proposition}

\begin{figure}[tbp] 
  \centering
  \includegraphics[bb=0 -1 538 296,width=3.92in,height=2.16in,keepaspectratio]{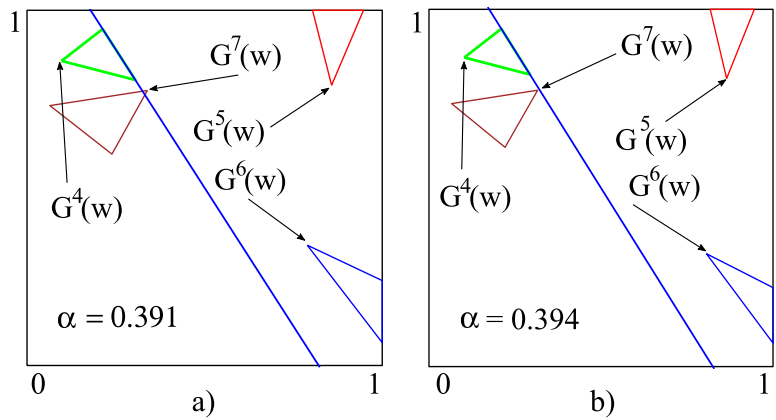}
  \caption{Further images of $G( G^3(B)\cap A_2)\cap A_1$ for a) $\alpha=0.391$ and b) $\alpha=0.394$}
  \label{fig:2after3D2_2i3nn}
\end{figure}


\begin{proof} Figure \ref{fig:2after3D2_1} shows the first four  images of $B=G(A_2)\cap A_1$ (green thick boundary). The blue line is the partition line $S(x,y)=1/2$.
The images are consecutively   $G(B)$ (red), $G^2(B)$ (blue), $G^3(B)$ (brown).
The set $ G^3(B)\cap A_2$ is bounded by thick brown lines and represents points which stay in $A_2$ for 3 steps. Its image is bounded by green lines. 
The set we are interested in is the triangle $C=G( G^3(B)\cap A_2)\cap A_1$, namely
the points which after three steps in $A_2$ go to $A_1$.

Further images of the triangle $C$  are shown in Figure \ref{fig:2after3D2_2i3nn} for a) $\alpha=0.391$ and b) $\alpha=0.394$. The important point is $G^7(w)$ (the same point $w$ as in the proof of Proposition \ref{1times}).
When  $G^7(w)\in A_2$, then some points of $C$  stay in $A_2$ longer than twice. When  $G^7(w)\in A_1$,  all points of $C$  stay in $A_2$ exactly  for two steps.
Equation $S(G^7(w))=1/2$ is equivalent to $192\alpha^7+192\alpha^6-336\alpha^5-144\alpha^4+256\alpha^3-128\alpha^2+53\alpha-11=0$ with a root $\alpha\sim 0.3931078326$.
Since $0.3931078326<0.3938896523$ for $\alpha > 0.3931078326$ we replace estimate 1) of Proposition \ref{estimates2}  with estimate of Proposition \ref{0.416} which holds up to $\alpha\sim 0.4160029431$.
\end{proof}

\begin{proposition}\label{no3no4} For $\alpha> 0.3510763028$  group $D_1D_2^4$ is not admissible.
 For $\alpha> 0.4284630893$  group $D_1D_2^3$ is not admissible. The following estimates hold:

1) $ \sigma_2(D_1D_2^2D_1D_2^2D_1D_2^3)>1$  at least for $\alpha\le  0.4315221884$;

2) $ \sigma_2(D_1D_2^2D_1D_2^3D_1D_2^2D_1D_2^3)>1$  at least for $\alpha\le  0.4584009011$;

3) $ \sigma_2(D_1D_2^4D_1D_2^2D_1D_2^3)>1$   for all $\alpha<0.5$. Although the group $D_1D_2^4$ may not be admissible, this inequality can be used for estimates.

4) $ \sigma_2(D_1D_2^5D_1D_2^2D_1D_2^3)>1$  at least for  $\alpha\le 0.4456891654 $;

5) $ \sigma_2(D_1D_2^6D_1D_2^2D_1D_2^3)>1$  at least for  $\alpha\le 0.4624281766 $.

For  $n=3k+i$,  $i=4,5,6$, we have
\begin{equation}\label{ineqsfront}\sigma_2(D_1D_2^nD_1D_2^2D_1D_2^3)=\sigma_2(D_2D_2^nD_1D_2^2D_1D_2^3)\ge \sigma_2^k(D_2^3)\sigma_2(D_1D_2^iD_1D_2^2D_1D_2^3).\end{equation}
With the previous results this extends the interval of the existence of acim up to $\alpha\sim 0.4345268819$ by estimate 3) of Proposition \ref{estimates2}).
\end{proposition}

\begin{figure}[h] 
  \centering
  \includegraphics[bb=0 -1 676 356,width=3.92in,height=2.07in,keepaspectratio]{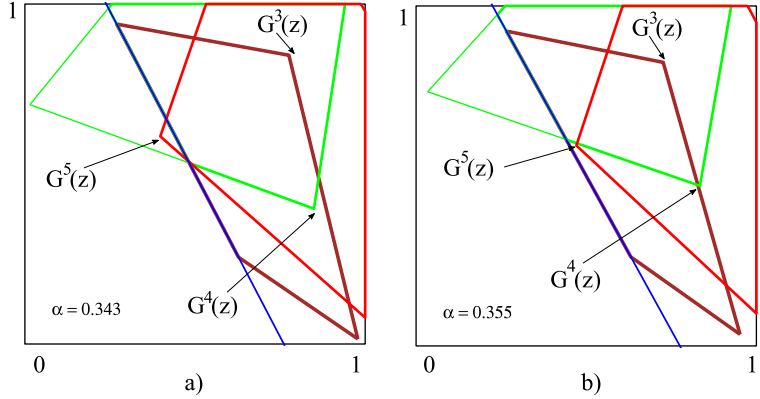}
  \caption{Further images of $C_1=G( G^3(B)\cap A_2)\cap A_2$ (thick brown), for a) $\alpha=0.343$ and b) $\alpha=0.355$.}
  \label{fig:no4_1i2}
\end{figure}

\begin{proof} First, the estimates 1)--5) show that the basic admissible sequences starting with $D_1D_2^3$ (followed by $D_1D_2^2$ 
in view of Proposition \ref{0.416}) have $\sigma_2>1$
up to $\alpha\sim  0.4315221884$. 

 Now, we will show that groups $D_1D_2^4$ and $D_1D_2^3$ are not admissible above some $\alpha$'s.
Figure \ref{fig:no4_1i2} shows further images of $C_1=G( G^3(B)\cap A_2)\cap A_2$ (thick brown), where $B=G(A_2)\cap A_1$ (thick green) 
shown in Figure \ref{fig:2after3D2_1}. The first image  $G(C_1)$ is bounded in green. These are points which were 3 steps
in $A_2$, some of them are in $A_1$, some stay for the fourth step in $A_2$. The region bounded in red is the image
$G(G(C_1)\cap A_2)$ (thick green), the points which were in $A_2$ for 4 steps. For $\alpha=0.343$ ( a)) some of them land in $A_1$, for 
$\alpha=0.355$ ( b)) the whole image is in $A_2$. The important point is $G^5(z)$, where $z=(0,2\alpha)$ is a vertex of $B$.
Equation $S(G^5(z))=1/2$ is equivalent to $\alpha^6+8\alpha^5-8\alpha^4-40\alpha^3-48\alpha^2-96\alpha+320=0$ with a root $\alpha\sim 0.3510763028$.

\begin{figure}[h] 
  \centering
  \includegraphics[bb=0 0 681 335,width=3.92in,height=1.93in,keepaspectratio]{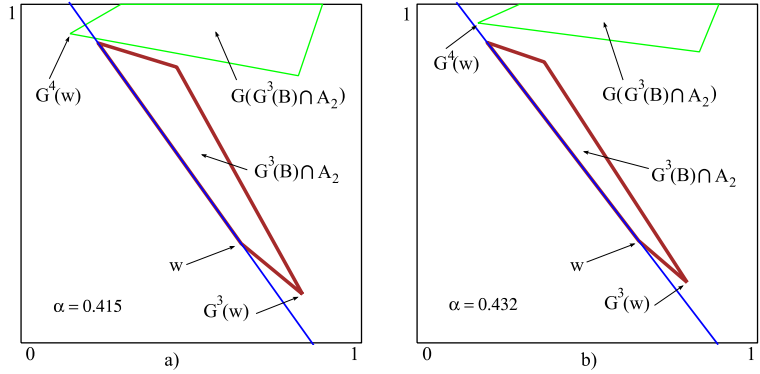}
  \caption{The image of $G^3(B)\cap A_2$ for a) $\alpha=0.415$ and b)  $\alpha=0.432$. }
  \label{fig:no3_1i2}
\end{figure}

Figure \ref{fig:no3_1i2} shows the set $G^3(B)\cap A_2$ (thick brown), the set of point which stayed in $A_2$ for three steps.
 $B=G(A_2)\cap A_1$ as in the proof
 of Proposition \ref{0.416} and point $w$ is also the same as there. The image $G(G^3(B)\cap A_2)$ is bounded in green.
The important point is $G^4(w)$. When $G^4(w)\in A_1$, then some points can go to $A_1$ after three steps in $A_2$. When
$G^4(w)\in A_2$, then all points which stayed 3 times in $A_2$ stay there for at least two more steps (4 times in $A_2$ were excluded 
in the previous part of the proof). The equation $S(G^4(w))=1/2$ is equivalent to
$24\alpha^4+12\alpha^3-36\alpha^2+9\alpha+1=0$ with  a root $\alpha\sim 0.4284630893$.

Once the the sequence $D_1D_2^3$ is rendered inadmissible, the worst estimate is $\alpha\sim 0.4345268819$, estimate 3) of Proposition \ref{estimates2}.
\end{proof}

To further improve the range of $\alpha$'s for which $G$ has an acim we have to consider  sequences starting with sequence
$D_1D_2^6$.

\begin{figure}[h] 
  \centering
  \includegraphics[bb=0 -1 322 330,width=3.92in,height=4.02in,keepaspectratio]{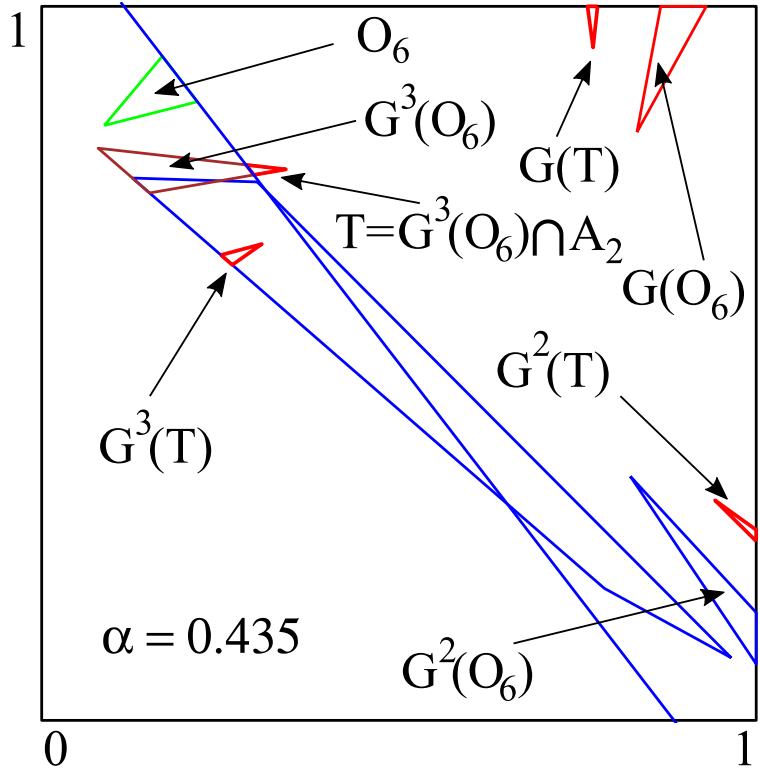}
  \caption{Images of points which stayed for 6 steps in $A_2$.}
  \label{fig:2or5after6}
\end{figure}

\begin{proposition} \label{2or5after6} Above $\alpha\sim 0.4345268819$ the sequence  $D_1D_2^6$ is followed by the sequence $D_1D_2^2$ or $D_1D_2^5$.
After $\alpha\sim 0.4397492527 $ the only possibility is $D_1D_2^2$.
\end{proposition}

\begin{proof} The blue quadrangle in Figure \ref{fig:2or5after6} is $G^3(G^3(B)\cap A_2)$, i.e., it is the third image of brown quadrangle of
 Figure \ref{fig:no3_1i2}. These are images of points which (for our range of $\alpha$'s) were for 5 steps in $A_2$.
The green triangle $O_6=G(G^3(G^3(B)\cap A_2)\cap A_2)\cap A_1$ are the points which went to $A_1$ after 6 steps in $A_2$.
Figure  \ref{fig:2or5after6} shows the images $G(O_6)$ (bigger red), $G^2(O_6)$ (blue) and $G^3(O_6)$ (partially brown, partially red). 
The points in $G^3(O_6)\cap A_1$  (brown part of $G^3(O_6)$) correspond to group $D_1D_2^2D_1D_2^6$.
Figure  \ref{fig:2or5after6} shows also three consecutive  images of $T=G^3(O_6)\cap A_2$ (small red triangles). In particular
$G^3(T)$ is completely inside $A_1$. These points correspond to the group $D_1D_2^5D_1D_2^6$. This proves the first claim of the proposition.

\begin{figure}[h] 
  \centering
  \includegraphics[bb=0 -1 669 339,width=3.92in,height=1.99in,keepaspectratio]{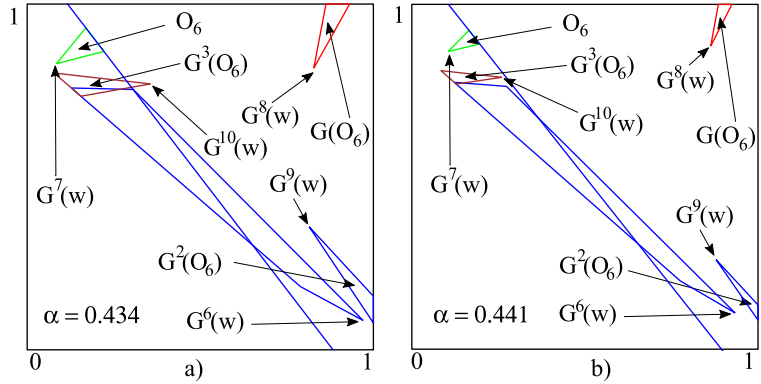}
  \caption{When the sequence $D_1D_2^5D_1D_2^6$ becomes inadmissible.}
  \label{fig:5endsafter6_1i2}
\end{figure}

Figure \ref{fig:5endsafter6_1i2} shows $O_6$ and its images  $G(O_6)$, $G^2(O_6)$  and $G^3(O_6)$ for parameters
$\alpha=0.434$ (part a)) and  $\alpha=0.441$ (part b)). For larger $\alpha$'s the image $G^3(O_6)$ is completely in $A_1$, which
means that after group $D_1D_2^6$ there must be group $D_1D_2^2$. The group $D_1D_2^5D_1D_2^6$ is no longer admissible.
The important point is $G^{10}(w)$, where $w$ is the point used already in Propositions \ref{no3no4} and \ref{0.416}.
The equation $S(G^{10}(w)=1/2$ is equivalent to $1536\alpha^{10}+3840\alpha^9-2688\alpha^8-7296\alpha^7+4128\alpha^6+3840\alpha^5-3504\alpha^4+992\alpha^3-160\alpha^2-5\alpha+11=0$ with a root $\alpha\sim 0.4397492527$.
\end{proof}

\begin{figure}[h] 
  \centering
  \includegraphics[bb=0 -1 669 340,width=3.92in,height=1.99in,keepaspectratio]{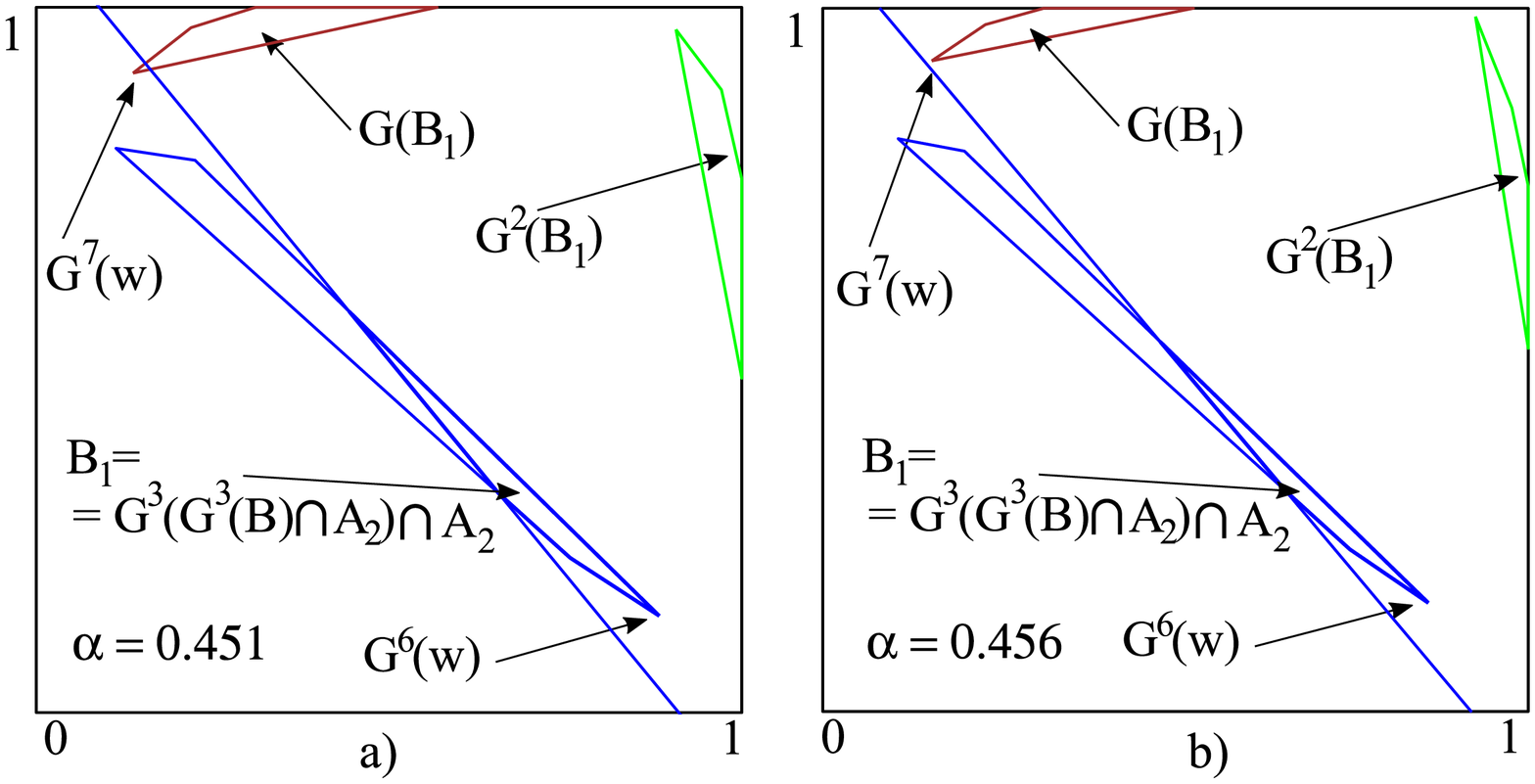}
  \caption{Sequence $D_1D_2^6$ becomes inadmissible.}
  \label{fig:n6or7_1i2}
\end{figure}

\begin{proposition} Above $\alpha\sim 0.4546258153$  the sequence $D_1D_2^6$ becomes inadmissible. For this range of $\alpha$
the sequence $D_1D_2^7$ is also inadmissible.
\end{proposition}

\begin{proof} Figure \ref{fig:n6or7_1i2} shows the quadrangle $B_1=G^3(G^3(B)\cap A_2)\cap A_2$ (thick blue), the set of points which
stay in $A_2$ for 6 steps. The images $G(B_1)$ (brown)  and $G^2(B_1))$ (green) are also shown. Part a) is for 
$\alpha=0.451$ and part b) for $\alpha=0.456$. For larger $\alpha$ both images
are completely inside $A_2$. This means that the sequences $D_1D_2^6$ and $D_1D_2^7$ are inadmissible.
The important point is $G^7(w)$ for the same point $w$ as before. The equation $S(G^7(w))=1/2$ is 
equivalent to $192\alpha^7+384\alpha^6-432\alpha^5-480\alpha^4+480\alpha^3-69\alpha+11=0$ with a root $\alpha\sim 0.4546258153$.
\end{proof}

\begin{proposition} \label {0.452} We have proved the existence of acim for $alpha$'s up to $\alpha\sim 0.4345268819$ 
(Proposition \ref{no3no4}). We have the following estimates on the $\sigma_2$'s of sequences starting with $D_1D_2^6$:

1) $ \sigma_2(D_1D_2^5D_1D_2^6)>1$  at least for $\alpha\le  0.4487890698$;

2) $ \sigma_2(D_1D_2^2D_1D_2^6)>1$  at least for $\alpha\le  0.4451846371$;

3) $ \sigma_2(D_1D_2^2D_1D_2^2D_1D_2^6)>1$  at least for $\alpha\le  0.4527916100$;

4) $ \sigma_2(D_1D_2^4D_1D_2^2D_1D_2^6)>1$   for all $\alpha<0.5$. Although the group $D_1D_2^4$ maybe not admissible, this inequality can be used for useful estimates.

5) $ \sigma_2(D_1D_2^5D_1D_2^2D_1D_2^6)>1$  at least for  $\alpha\le 0.4600595036 $;

6) $ \sigma_2(D_1D_2^6D_1D_2^2D_1D_2^6)>1$  at least for  $\alpha\le 0.4718920017 $.

These estimates and previous results extend the range of the existence of acim up to $\alpha\sim 0.4527916100$.
\end{proposition}

\begin{proof} Estimate 1) together with Proposition \ref{2or5after6}   tell us that all sequences starting with $D_1D_2^5D_1D_2^6$ have $\sigma_2>1$ as long as they are admissible.
All other sequences starting with $D_1D_2^6$ must  start with $D_1D_2^2D_1D_2^6$. Using inequality (\ref{ineqsfront}) and estimates 2)--6) we see that they all have $\sigma_2>1$
at least up to $\alpha\sim 0.4527916100$.
\end{proof}

\begin{figure}[h] 
  \centering
  \includegraphics[bb=0 -1 669 339,width=3.92in,height=1.99in,keepaspectratio]{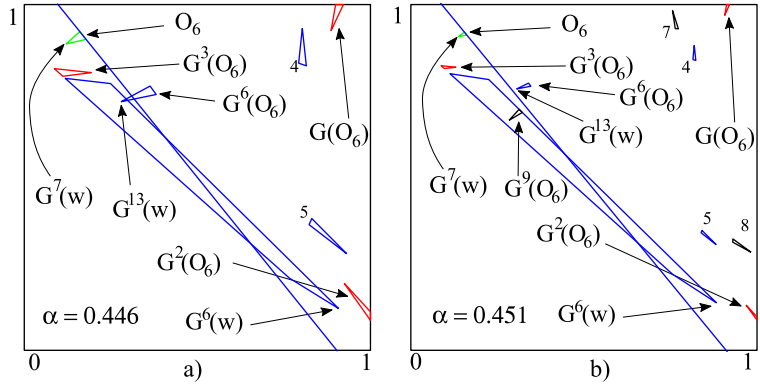}
  \caption{Images of $O_6$: a) 6 images for $\alpha=0.446$, b) 9 images for $\alpha=0.451$.}
  \label{fig:2and5after6_1i2}
\end{figure}

We want to push $\alpha$ higher to make the sequences starting with $D_1D_2^6$ inadmissible. First, we will find out
what comes after the sequence $D_1D_2^2D_1D_2^6$ for $\alpha>0.4527916100$.

\begin{proposition} \label{2andthen5after6} After $\alpha\sim 0.4496432201 $ after the sequence $D_1D_2^6$ comes the sequence $D_1D_2^5D_1D_2^2$.
\end{proposition}

\begin{proof} Figure \ref {fig:2and5after6_1i2} a) shows 6 consecutive images of triangle $O_6$ (introduced in Proposition \ref{2or5after6}), the set of points which leave $A_2$ after staying in it for six steps, for $\alpha=0.446$. 
The triangle $G^3(O_6)$ is completely in $A_1$. This corresponds to the sequence $D_1D_2^2D_1D_2^6$, whose necessity was proved in Proposition \ref{2or5after6}.  The triangle
$G^6(O_6)$ intersects the partition line so some points leave $A_2$ at this moment, some continue staying in $A_2$.

Part b) of the same figure show the same 6 images of $O_6$ and 3 next images, for $\alpha=0.451$. Some images have full descriptions, some only numbers.  For this $\alpha$ triangle $G^6(O_6)$ is completely inside $A_2$ so all of its points continue
staying in $A_2$. The triangle $G^9(O_6)$ is completely in $A_1$. This shows that for this range of $\alpha$'s after
group $D_1D_2^6$ there must be group $D_1D_2^5D_1D_2^2$.

The important point is $G^{13}(w)$ (the same $w$ as before), the left most vertex of $G^6(O_6)$. The equation
$S(G^{13}(w))=1/2$ implies
$12288\alpha^{13}+36864\alpha^{12}-12288\alpha^{11}-86016\alpha^{10}+16128\alpha^9+84480\alpha^8-43392\alpha^7-23360\alpha^6+36288\alpha^5-19456\alpha^4+2816\alpha^3+1984\alpha^2-869\alpha+91=0$ with a root $\alpha\sim 0.4496432201 $.
\end{proof}

\begin{theorem} The map $G$ admits an acim for $\alpha$ up to at least $\alpha\sim 0.4600595036$.
\end{theorem}

\begin{proof} In Proposition \ref{0.452} we proved existence of an  acim up to $\alpha\sim  0.4527916100$. After this value, by Proposition \ref{2andthen5after6} the offending sequence $D_1D_2^2D_1D_2^2D_1D_2^6$ is no longer admissible.
The lowest estimate we need now is estimate 5) of Proposition \ref{0.452}. Thus, the existence of an acim is proved for $\alpha$'s
up to $\alpha\sim 0.4600595036$.
\end{proof}

\begin{remark} For  $\alpha$'s above $\alpha\sim 0.4600595036$ the sequence $D_1D_2^6$ is no longer admissible.
\end{remark}

The exact estimates for $\alpha>0.4600595036$ become more and more complicated. We hope to find some more abstract way to prove that $G$ satisfies the expanding conditions of \cite{Tsu}.
We performed numerical experiments estimating $\sigma_2(x_0,N)=\sigma_2 \left(\prod_{k=0}^N DG(G^k(x_0))\right)$ for millions of initial points $x_0$. Instead of calculating $\sigma_2$ directly, we used estimate (see, e.g., \cite{Zou})
\begin{equation}\label{Zou}\sigma_2\left(\prod_{k=0}^N M_k\right)\ge \frac {\det\left(\prod_{k=0}^N M_k\right)}{\left|\left|\prod_{k=0}^N M_k\right|\right|_F}=
\frac {\prod_{k=0}^N \det M_k}{\left|\left|\prod_{k=0}^N M_k\right|\right|_F} ,
\end{equation}
where $\|M\|_F=\sqrt{m_{1,1}^2+m_{1,2}^2+m_{2,1}^2+m_{2,2}^2}$ is the Frobenius norm of the matrix $M$. Since all $M_k$'s are either $D_1$ or $D_2$ and $\det D_1=\det D_2$, the calculations of right hand side of 
(\ref{Zou}) are very stable. All trials showed that for $\alpha<1/2$ the quantity $\sigma_2(x_0,N)$ grows to infinity as $N$ increases. This provides numerical evidence for expanding properties of $G$ and
the existence of acim.

\begin{figure}[h] 
  \centering
  \includegraphics[bb=0 -1 800 418,width=3.66in,height=1.91in,keepaspectratio]{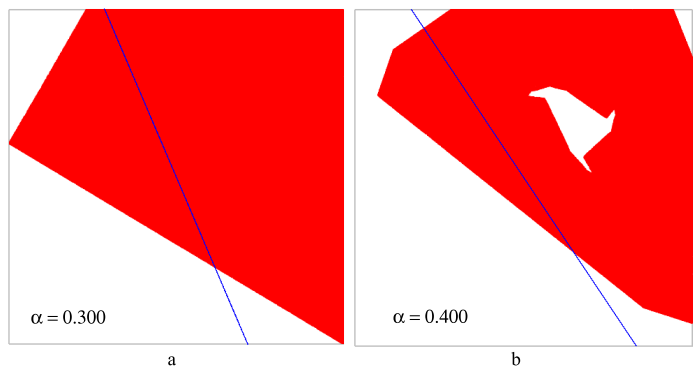}
  \caption{Support of acim for $\alpha=0.3$ and $\alpha=0.4$.}
  \label{fig:acim_03-04}
\end{figure}

The Figures \ref{fig:acim_03-04}--\ref{fig:acim_043-046} show the support of acim (or conjectured acim)
for $\alpha=0.3,0.4, 0.43, 0.46, 0.49 , 0.495$. The pictures were obtained by computer plotting $10^6$ iterates
long trajectory of $G_\alpha$ after skipping the first $1.5\cdot 10^6$ iterations. The experiments show that the obtained support is independent of the typical initial point.

\begin{figure}[h] 
  \centering
  \includegraphics[bb=0 -1 808 417,width=3.66in,height=1.89in,keepaspectratio]{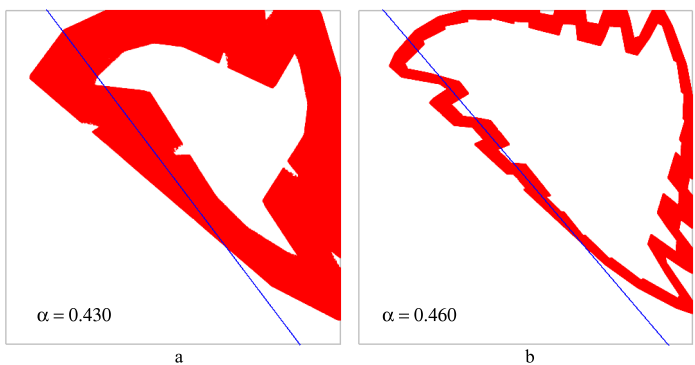}
  \caption{Support of acim for $\alpha=0.43$ and $\alpha=0.46$.}
  \label{fig:acim_043-046}
\end{figure}

\begin{figure}[h] 
  \centering
  \includegraphics[bb=0 -1 816 417,width=3.66in,height=1.87in,keepaspectratio]{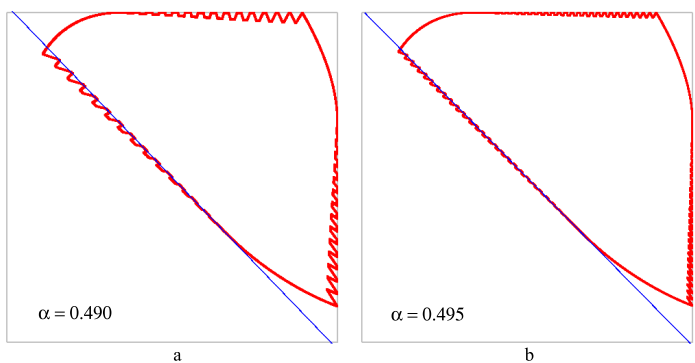}
  \caption{Support of conjectured acim for $\alpha=0.49$ and $\alpha=0.495$.}
  \label{fig:acim_049-0495}
\end{figure}

\begin{figure}[tbp] 
  \centering
  \includegraphics[bb=0 -1 800 417,width=3.66in,height=1.91in,keepaspectratio]{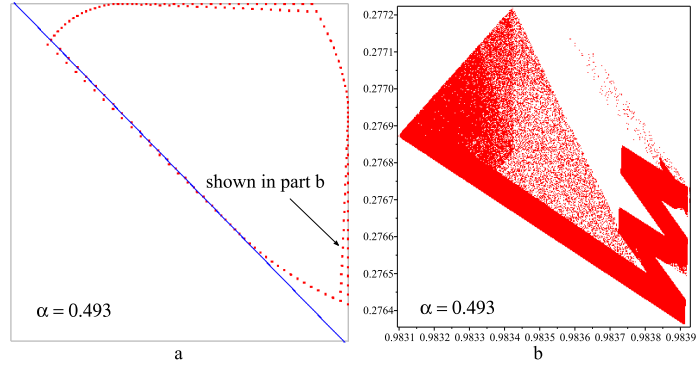}
  \caption{ a: Support of conjectured acim for $\alpha=0.493$. b: Close-up of one of the clusters in part a. }
  \label{fig:acim_0493a}
\end{figure}

For $\alpha$'s in a very narrow window around $\alpha=0.493$ (of radius approximately $10^{-6}$), the support of conjectured acim looks very different from typical, see Figure \ref{fig:acim_0493a}.
It consists of 175 clusters which under action of $G$ move by 58 positions in the clockwise direction. Since $3\cdot 58=174$, $G^{175}$ preserves every  cluster.
Figure  \ref{fig:acim_0493a} b shows one of the clusters (pointed out by an arrow in part a). It shows $500\cdot 10^6$ iterations of $G^{175}$, after skipping $35\cdot 10^6$ initial iterations. Parts of the image were showing up
extremely slowly.
We observed similar behaviour for  $\alpha=0.4883$ (106 clusters moving by 35 positions), $\alpha=0.4943$ (214 clusters moving by 71 positions) and $\alpha=0.4973$ (448 clusters moving by 149 positions). Probably there are many other windows of $\alpha$ with similar behaviour.

\section{Deterministic Behaviour of Memory Map for $1/2\le \alpha\le 3/4$}\label{sec:deterministic}

\subsection{
$ \alpha=1/2$}

 Let $ \alpha=1/2$. In particular, we have $$\tau(1-x/2)=2-2(1-x/2)=x\, .$$
Assume $y\ge 1-x$ or $x+y\ge 1$ or $(x+y)/2\ge 1/2$. Then,
\begin{equation} \begin{split}G(x,y)&=(y, \tau((x+y)/2))=(y,2-x-y)\, ,\\
G^2(x,y)&=G(y,2-x-y)=(2-x-y, \tau(1-x/2))=(2-x-y,x)\, ,\\
G^3(x,y)&=G(2-x-y,x)=(x,\tau(1-y/2))=(x,y)\, .
\end{split}
\end{equation}
This shows that any such point is periodic with period 3. The only fixed point in this region is $(2/3,2/3)$. (Another one is $(0,0)$ and there is no more fixed points)

If $y<1-x$, then we have to show that any such point except $(0,0)$ eventually goes to the upper triangle $y\ge 1-x$.
Note that if $G(x,y)=(0,0)$, then $(x,y)=(0,0)$. Also, $G(x,0)=(0,\tau(x/2))$, so we can consider only points with $y>0$.
Then, as long as the second coordinate is less than 1 minus the first, we have 
$$(x,y)\mapsto (y,x+y)\mapsto (x+y,x+2y)\mapsto (x+2y,2x+3y)\mapsto \dots $$
It is clear that the sum of the coordinates grows on each step at least by the value $y$ so eventually it goes above 1,
which means that the point goes to the upper triangle.

\subsection{
$ \alpha=3/4$}

Let $\alpha=3/4$ and let us assume that $x+y=4/3$ or $3x+3y=4$. Then,
$$G(x,y)=\left(y,\tau\left(\frac 34 y+ \frac 14 x\right)\right)\, .$$
We have $\frac 34 y+ \frac 14 x=\frac 14 (3y+3x-2x)=1-x/2\ge 1/2$ so
$$\tau\left(\frac 34 y+ \frac 14 x\right)=2-\frac 14(6y+6x-4x)=2-2+x=x\, .$$
Thus, for such points $$G(x,y)=(y,x),$$ so each of them is periodic with period 2, except for the fixed point $(2/3,2/3)$.

We will prove the following:

\begin{theorem}\label{Th075}
 For $\alpha=3/4$ any  point, except $(0,0)$ is either periodic (period 2 or 1) or eventually periodic or attracted to the line $x+y=4/3$.
\end{theorem}

The line $y=2/3-x/3$ (or equivalently $y+x/3=2/3$) partitions square $[0,1]\times [0,1]$ into two parts on which $G$ is defined differently: $A_1$ below the line and $A_2$ above it.
We partition region $A_2$ further into three parts: $B_1$ between the lines $y=-x/2+5/6$ and $y=-x/2+7/6$, $B_2$ between $y=-x/2+5/6$
 and the partition line and $B_3$ above the line $y=-x/2+7/6$, see Figure \ref{fig:partitions075}.

\begin{figure}[h] 
  \centering
  \includegraphics[bb=20 118 575 673,width=3.92in,height=3.92in,keepaspectratio]{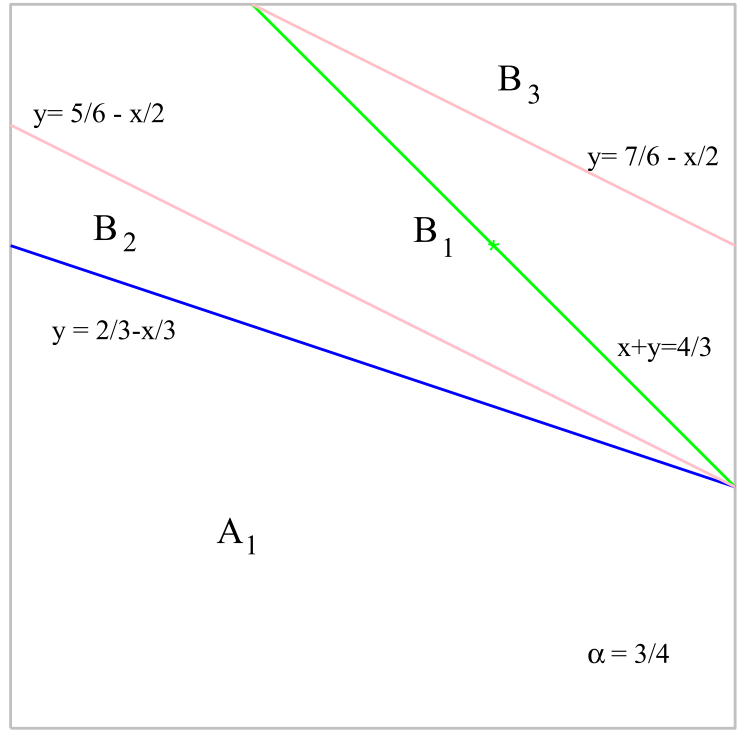}
  \caption{Regions for $\alpha=3/4$.}
  \label{fig:partitions075}
\end{figure}

Let $(x,y)\in A_1\setminus \{(0,0)\}$. If $(x,y)=(x,0)$, then $G(x,0)=(0,x/2)$, so we can assume that
$y>0$. 
 It is easy to calculate that 
$$G(x,y)=\left(y,\frac 32 y+\frac 12 x\right),$$
with the sum of second coordinate plus one third of the first coordinate equal to $\frac 52 y+\frac 16 x$ so on each step this sum grows by at least $y$ and eventually every such point will move to the
upper half of the square $y+x/3> 2/3$.


Consider now the region $B_{1}$ inside $\left[ 0,1\right] \times \left[ 0,1%
\right] $ between the lines $y=-x/2+5/6$ and $y=-x/2+7/6.$ \ It contains the
line $L:x+y=4/3$ of periodic points. \ The derivative matrix in this region
is constant and has eigenvalues $-1,-1/2$ and corresponding eigenvectors $%
v_{1}=\left[ -1,1\right] $ and $v_{2}=\left[ -2,1\right] .$ Every point in $%
B_{1}$ can be written uniquely as $\left[ 2/3,2/3\right] +tv_{1}+sv_{2}=%
\left[ 2/3-t-2s,2/3+t+s\right] $ for $\left[ t,s\right] \in E,$ some compact
neighbourhood of $\left[ 0,0\right] $. \ We have%
\begin{eqnarray*}
G\left( \left[ 2/3-t-2s,2/3+t+s\right] \right)  &=&\left[ 2/3+t+s,2/3-t-s/2%
\right]  \\
&=&\left[ 2/3,2/3\right] -tv_{1}-s/2v_{2}
\end{eqnarray*}%
and since $v_{1}$ is parallel to $L,$ this means the distance to $L$ is
divided by 2.Thus, every point in $B_1$ is attracted to the periodic line.

Let us consider $B_2$ now.
 We will show that $G(B_2)\subset B_3$.
Let $(x,y)\in B_2$. Then, $y<5/6-x/2$ and $G(x,y)=(w,z)=(y, 2-(3/2)y-(1/2)x)$.
We will show that $z> 7/6 -w/2$, or 
$$2-\frac 32 y-\frac 12 x>\frac 76-\frac 12 y\, ,$$
which is exactly our assumption. Thus, $G(B_2)\subset B_3$.

\begin{figure}[h] 
  \centering
  \includegraphics[bb=0 -1 667 339,width=3.92in,height=2in,keepaspectratio]{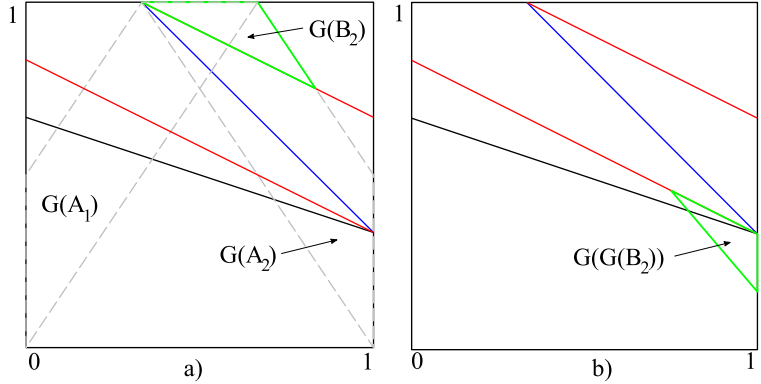}
  \caption{Images $G(B_2)$ and $G(G(B_2))$, $\alpha=3/4$.}
  \label{fig:075imagesGGB2aib}
\end{figure}

In Figure \ref{fig:075imagesGGB2aib} a) we see the image $G(B_2)$ (green) and both images $G(A_1)$ and $G(A_2)$ (grey dashed). The points outside $G(A_1)\cup G(A_2)$ are transient and
unimportant for dynamics because they are eventually mapped into $G(A_1)\cup G(A_2)$. Thus, the only part of $B_3$ we will study is the image $G(B_2)$. In  Figure \ref{fig:075imagesGGB2aib} b) we see the image $G(G(B_2))$ (green). It consists of two parts,
upper $G^2(B_2)\cap A_2$ and lower  $G^2(B_2)\cap A_1$.

\begin{figure}[tbp] 
  \centering
  \includegraphics[bb=0 -1 666 339,width=3.92in,height=2in,keepaspectratio]{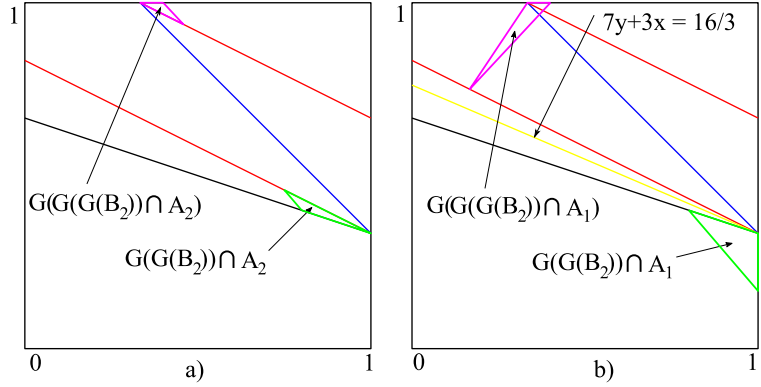}
  \caption{Images of a) the upper part and b) the lower part of $G(G(B_2))$}
  \label{fig:075imagesGGGB2aib}
\end{figure}

In Figure \ref{fig:075imagesGGGB2aib} a) we see the image $G(G^2(B_2)\cap A_2)$ (magenta) of the upper part of $G^2(B_2)$. We have $G(G^2(B_2)\cap A_2)\subset G(B_2)\subset A_2$ so further iterations of these points will be 
similar to that of whole $G(B_2)$. In Figure \ref{fig:075imagesGGGB2aib} b) we see the image $G(G^2(B_2)\cap A_1)$ (magenta) of the lower part of $G^2(B_2)$. We see that the points of $G(G^2(B_2)\cap A_1)$ are either
in $B_1$ (and then their future iterates are attracted to the line $x+y=4/3$) or they are inside $G(B_2)$ above the line $y=-x/2+7/6$ (upper red). The lowest point of $G(G^2(B_2)\cap A_1)$ is $(1/6,3/4)$ and belongs to the line
$y=-x/2+5/6$ (lower red).

 Under the action of $G$ every point in $A_2$ gets closer to the line $x+y=4/3$ (blue). To show that every point of $G(B_2)$ is attracted to this line, it is enough to show that for any point $(x,y)\in  G^2(B_2)\cap A_1$
its image $(z,w)=(y, \frac 32 y+\frac 12 x)$ is either in $B_1$ or  is closer to the line $x+y=4/3$  than $(x,y)$.
Using the  formula for the distance of a point from a line we have to check that
$$|x+y-4/3|>|z+w-4/3|.$$
Since the point $(x,y)$ is below the partition line we have $|x+y-4/3|=4/3-x-y$. Since the point $(z,w)$ is above line  $y=-x/2+7/6$ (upper red) we have $|z+w-4/3|=z+w-4/3$. Thus, our condition is equivalent to
$4/3-x-y>z+w-4/3$, or
\begin{equation}\label{line yellow}
4/3-x-y>y+\frac 32 y+\frac 12 x-4/3,\ \text{or}\ y<-\frac 37 x+\frac {16}{21} . 
\end{equation}
The line $y=-\frac 37 x+\frac {16}{21}$ (yellow) intersects the partition line $y=-\frac 13 x +\frac 23$ at the point $(1,1/3)$ and for $x\in (0,1)$ is above it. Thus, all points in $G^2(B_2)\cap A_1$ satisfy the condition
(\ref{line yellow}). This proves Theorem \ref{Th075}.


\subsection{$1/2<\alpha<3/4$}

Let $1/2<\alpha<3/4$. We will prove that the fixed point $x_0=(2/3,2/3)$ is the global attractor attracting all points except $(0,0)$. The derivative matrix at $x_0$ is 
$$ D=\left[ \begin{matrix} 0& 1\\ -2(1-\alpha)& -2\alpha \end{matrix}\right]\, ,$$
with eigenvalues $e_1=-\alpha+\sqrt{\alpha^2+2\alpha-2}$, $e_2=-\alpha-\sqrt{\alpha^2+2\alpha-2}$ which are complex for $1/2<\alpha<\sqrt{3}-1$ and real for $\sqrt{3}-1\le \alpha< 3/4$.
In the interval $(1/2,\sqrt{3}-1)$ their moduluses are equal to $|e_1|=|e_2|=\sqrt{2(1-\alpha)}$ and
less than 1. In the interval $[\sqrt{3}-1, 3/4)$ eigenvalue $e_2$ has larger modulus equal
$|e_2|=-e_2=\alpha+\sqrt{\alpha^2+2\alpha-2}$ also less than 1.
Thus, $x_0$ is an attracting fixed point. 

We will now prove  a few facts. Recall that  $A_1$ denote the part of the square $[0,1]\times [0,1]$ below the line $\alpha y+(1-\alpha)x=1/2$ and $A_2$ the part above this line.

We extend Proposition \ref{pr1} to :
 \begin{proposition} \label {pr1a} If $(x,y)\in A_1$ and $\alpha y+(1-\alpha)x> a $, $a<1/2$, then the point $(w,z)=G(x,y)$ satisfies $\alpha z+(1-\alpha)w> 2\alpha a$, holds also for the $1/2<\alpha<3/4$.
\end{proposition} 

\begin{proposition}\label{pr052} If $(x,y)\in A_2$, then the point $(w,z)=G(x,y)$ satisfies $\alpha z+(1-\alpha)w\ge 1-\alpha $. 
\end{proposition}
 
\begin{proof} We have $$\alpha z+(1-\alpha)w= \alpha(2-2\alpha y-2(1-\alpha)x)+(1-\alpha)y= 2\alpha-[2\alpha^2+\alpha-1]y-2\alpha(1-\alpha)x\, .$$
The inequality $$ 2\alpha-[2\alpha^2+\alpha-1]y-2\alpha(1-\alpha)x\ge 1-\alpha\, ,$$
is equivalent to $$[2\alpha^2+\alpha-1]y+2\alpha(1-\alpha)x\le 3\alpha-1\, .$$ For $\alpha>1/2$ the left hand side of the inequality is an increasing function of $x$ and $y$
with maximum at $(1,1)$ equal to $3\alpha-1$.
 This completes the proof.
\end{proof}

Let $A'_I$ denote the part of  the square $[0,1]\times [0,1]$ above the line $\alpha y+(1-\alpha)x=1-\alpha$.
 Propositions \ref{pr1} and \ref{pr052} prove that $G(A'_I)\subset A'_I$, i.e., the region $A'_I$ is $G$-invariant.
It follows from Proposition \ref{pr1} that every point of $A_1$, except $(0,0)$, 
enters $A'_I$ after a finite number of steps.

\begin{proposition}\label{pr053} For every $(x,y)\in A_1\cap A'_I$ we have $G(x,y)\in A_2$ or $G^2(x,y)\in A_2$.
\end{proposition}

\begin{proof} Applying Proposition \ref{pr1a}  twice and  Proposition \ref{pr052} for  $a=1-\alpha$, it is enough to show that $$(2\alpha)^2(1-\alpha)>1/2\, .$$
Let $f(\alpha)=\alpha^2(1-\alpha)$. It is easy to check that on interval $[1/2,3/4]$ function $f$ is concave with maximum at $2/3$. We have $f(1/2)=\frac 14\cdot\frac 12=\frac 18$.
Also, $f(3/4)=\frac 9{16}\cdot\frac 14=\frac 9{64}>\frac 18$. This completes the proof.
\end{proof}


\begin{proposition}\label{pr05-0593} For $1/2<\alpha\le(\sqrt{33}-1)/8 \sim 0.5930703309$, the fixed point $X_0=(2/3,2/3)$ attracts all points except $(0,0)$.
\end{proposition}

\begin{proof}
We will construct a trapping region $T\subset A_2$, containing $X_0$, such that $G(T)\subset T$. Every point whose trajectory stays in $A_2$ is attracted to $X_0$, since $G_{|A_2}$ is an affine map with an attracting point
$X_0$. We will prove that every point of $A_2$ eventually enters $T$. From Proposition \ref{pr1} we know that every point except $(0,0)$ eventually enters $A_2$.

\begin{figure}[tbp] 
  \centering
  \includegraphics[bb=0 -1 659 343,width=3.92in,height=2.04in,keepaspectratio]{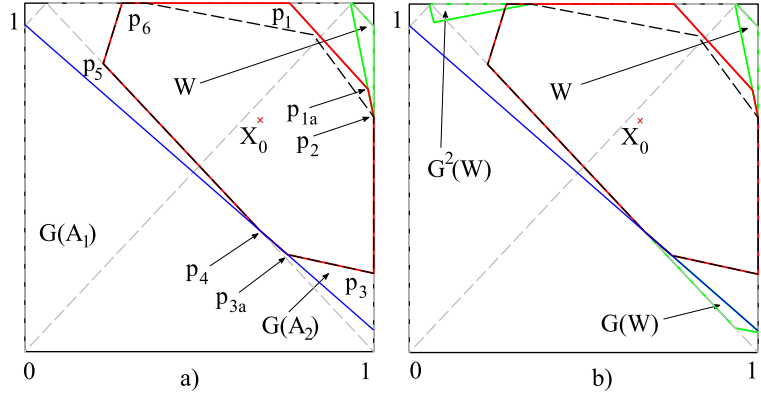}
  \caption{Trapping region $T$ for $1/2<\alpha\le \sim 0.593$. Case $\alpha=0.533$ is shown.}
  \label{fig:Trap050593aib}
\end{figure}

Construction of $T$: 
The trapping region $T$ is shown in Figure \ref{fig:Trap050593aib} a). It is a polygon with vertices $p_1,p_{1a}, p_2, p_3, p_{3a}, p_4, p_5$ and $p_6$ (red). Its image $G(T)$ is bounded by black dashed line. We will describe the choice of the vertices. Let $G_i=G_{|A_i}$, $i=1,2$. The large quadrangles bounded by dashed grey lines are the sets $G(A_1)$ and $G(A_2)$. We do not need to consider the points outside $G(A_1)\cup G(A_2)$ as they are transient and their images eventually go into trapping region or the region bounded by green lines. The green quadrangle (it looks like a triangle) is the set $W=G_2^{-1}(A_1)\cap G(A_2)$, the non-transient points of $A_2$ which go in one step to $A_1$.
Point $p_2$ is the lowest vertex of $W$. Then, consecutively $p_6=G_2^{-1}(p_2)$, $p_{3a}=G_2^{-1}(p_6)$ and $p_{1a}=G_2^{-1}(p_{3a})$. For the point $p_5$  we have $p_3=G_2^{-1}(p_5)$ and  $p_1=G_2^{-1}(p_3)$. The point $p_5$ is chosen on the boundary of $G(A_2)$ in such a way that its image $G(p_5)$ lies to the left of the line connecting $p_1$ and $p_{1a}$. Finally, $p_4$ is the intersection of the lower boundary of $G(A_2)$ and the partition line (blue). We also have $p_4=G(p_2)$.
By construction, every vertex of $T$ goes into $T$. Since $T$ is convex, we have $G(T)\subset T$.

The only thing we have to prove is that any point of $W$ (non-transient points going out of $A_2$) eventually enters the trapping region $T$. In Figure \ref{fig:Trap050593aib} b) we see that the second image $G^2(W)$ is a thin quadrangle (looking like a triangle) adjacent to the upper boundary of the square $[0,1]\times[0,1]$.  The lowest point of $G^2(W)$  is the point
$( 2\alpha(2\alpha-1),8\alpha^3-8\alpha+4)$. Its most to the right point is $(\alpha/(\alpha+1),1)$. We will prove in  Proposition
\ref{moving rect} that for  any point $(x,y)$ with  $x\le x_w=\alpha/(\alpha+1) $ and $y\ge y_w= 8\alpha^3-8\alpha+4$ and its third image $(z,w)=G^3(x,y)$ the difference $z-x$ is larger than some positive constant depending on $\alpha$ and $w\ge y_w$ unless $(z,w)\in T$.
This shows that any point of $G^2(W)$ eventually enters $T$, and completes the proof of Proposition \ref{pr05-0593}.
\end{proof}

\begin{proposition}\label{moving rect} Let $1/2<\alpha\le(\sqrt{33}-1)/8 \sim 0.5930703309$. Let point $(x,y)$ satisfies  $x\le
x_w= \alpha/(\alpha+1)$ and $y\ge y_w= 8\alpha^3-8\alpha+4$.  Then, for its third image $(z,w)=G^3(x,y)$ the difference $z-x$ is larger then some positive constant depending on $\alpha$. If $(z,w)\not\in T$, then $w\ge y_w$.
\end{proposition}

\begin{figure}[tbp] 
  \centering
  \includegraphics[bb=0 -1 628 332,width=3.92in,height=2.07in,keepaspectratio]{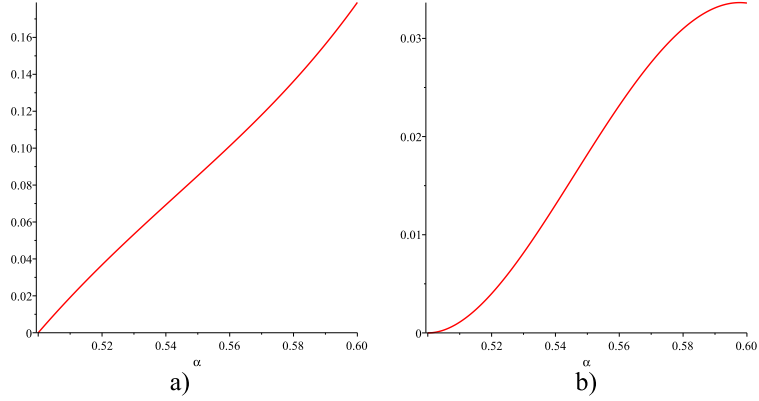}
  \caption{ a)The graph of $z-t$  and  b) of $y(z_i)-y_w$  for the proof of Proposition \ref{moving rect}}
  \label{fig:graphsforProp23}
\end{figure}

\begin{proof}
Let $(x,y)=(t,1-s)$ satisfy the assumptions. The third iterate  $G^3$ on such point is equal either $G_1\circ G_2\circ G_2$
or $G_2\circ G_2\circ G_2$. The first coordinate of $(z,w)=G^3(x,y)$ does not depend on the whether the last map applied is 
$G_1$ or $G_2$. We have $z-t= ct(\alpha) t+cs(\alpha) s + cc(\alpha) $, where
$$ ct=-4\alpha^2+4\alpha-1<0 \ ,\ cs=-4\alpha^2-2\alpha+2<0\ ,\ cc=2\alpha(2\alpha-1)>0 .$$
Since both $ct(\alpha)$ and $cs(\alpha)$ are negative  $z-t$ has the least value when both $t$ and $s$ are maximal,
i.e., $t=x_w$ and $s=1-y_w$. Then, $$z-t=2(2\alpha-1)(8\alpha^4+4\alpha^3-4\alpha^2-5\alpha+3)> 0 .$$ The graph of $z-t$ is shown in Figure \ref{fig:graphsforProp23}.

\begin{figure}[tbp] 
  \centering
  \includegraphics[bb=0 -1 717 340,width=3.92in,height=1.86in,keepaspectratio]{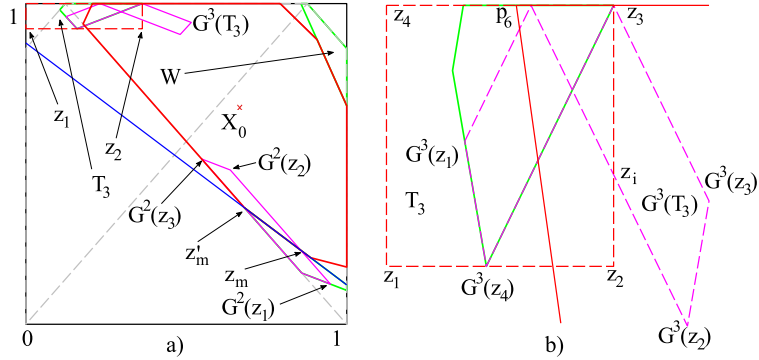}
  \caption{a) $T_3$ and its images, b) enlargement of $T_3$ and $G^3(T_3)$. }
  \label{fig:ImagesT3_aib}
\end{figure}

To prove the second claim  we will consider the images of the rectangle $T_3$ (see Figures
\ref{fig:graphsforProp23} b) and \ref{fig:ImagesT3_aib} ) with vertices
$z_1=(0,y_w)$, $z_2=(x_w,y_w)$, $z_3=(x_w,1)$ and $z_4=(0,1)$. The second image $G_2^2(T_3)$ has the vertices
$G^2(z_1),G^2(z_4)\in A_1$ and $G^2(z_2),G^2(z_3)\in A_2$. Its sides intersect partition line at points $z_m$ between
$G^2(z_1)$ and $G^2(z_2)$ and $z_m'$ between $G^2(z_3)$ and $G^2(z_4)$. The image $G^3(z_4)$ lies on the lower side of the rectangle
$T_3$ and the image $G^3(z_1)$ is higher.
The images $$G(z_m)=((8\alpha^4-12\alpha^3-6\alpha^2+17\alpha-6)/(\alpha+1),1)$$ and $G(z_m')$ are on the top side of the square.
The image 
\begin{equation*}
\begin{split}
G^3(z_2)&=\left(\frac {2}{\alpha+1}\left(16 \alpha^6+24 \alpha^5-16 \alpha^4-26 \alpha^3+12 \alpha^2+7 \alpha-3\right),\right.\\
&\left.-64 \alpha^6-64 \alpha^5+128 \alpha^4+40 \alpha^3-100 \alpha^2+36 \alpha-2\right).
\end{split}\end{equation*}

The line $L(G^3(z_2),G(z_m))$ intersects right hand side of $T_3$ at the point $z_i=(x_w,16\alpha^4-32\alpha^3+38\alpha-27+6/\alpha)$ 
with the second coordinate larger than $y_w$. This shows, that the points of $G^3(T_3)$ lie either in $T_3$ or in $T$.
Together with the first claim this shows that every point of $T_3$ eventually enters $T$.
\end{proof}


We continue to prove that the fixed point $(2/3,2/3)$ is  a global attractor for other intervals of parameter $\alpha\in (1/2,3/4)$.

\begin{proposition}\label{pr0593-0728} For $(\sqrt{33}-1)/8<\alpha\le \sqrt{33}/12+1/4\sim 0.7287135539$, the fixed point $X_0=(2/3,2/3)$ attracts all points except $(0,0)$.
\end{proposition}

\begin{proof} The general plan of the proof is the same as for Proposition \ref{pr05-0593}. We construct a trapping region $T$ and show that some (fourth or fifth) image of $W=G_2^{-1}(A_1)\cap G(A_2)$
falls into $T$.

\begin{figure}[h] 
  \centering
  \includegraphics[bb=0 -1 662 339,width=3.92in,height=2.01in,keepaspectratio]{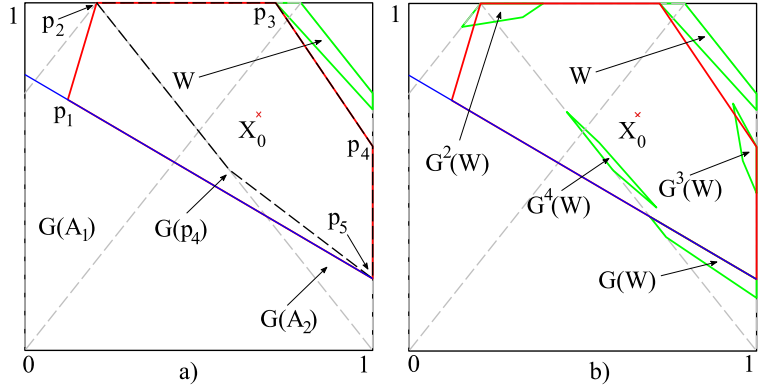}
  \caption{$\alpha =0.63$ (case ii)) a) Trapping region $T$ (red) and its image $G(T)$ (dashed black). b) Region $W$ and its images, $G^4(W)\subset T$. }
  \label{fig:trap059-065_1i2}
\end{figure}

Construction of the trapping region $T$: $T$ is a pentagon with the vertices: $p_3$ which is the upper left vertex of $W$,
$p_5=G(p_3)$, $p_2=G(p_5)$, $p_4=G(p_2)$, and $p_1=G_2^{-1}(p_3)$. Since, for $\alpha$ in the considered interval,
$G(p_4)\in T$, we $G(T)\subset T$, i.e., $T$ is a trapping region. Figure \ref{fig:trap059-065_1i2} a) shows the
trapping region $T$ (red) and its image $G(T)$ (dashed black). The green quadrangle is $W=G_2^{-1}(A_1)\cap G(A_2)$.

Below, we will show that fifth or fourth image of $W$ is a subset of  $T$. We consider subintervals of $\alpha$.

i) $(\sim 0.5930703309,\sim 0.5970091680)$

$\alpha=(\sqrt{33}-1)/8 \sim 0.5930703309$ is the largest $\alpha$ for which the sides of $W$ and $G^3(W)$ which are on the line
$x=1$ intersect. For $\alpha\in (\sim 0.5930703309,\sim 0.5970091680)$,  $W$ and $G^3(W)$ still intersect (the highest vertex of 
$G^3(W)$ is in $W$). ($\sim 0.5970091680$ is a root of $16\alpha ^5-16\alpha ^3+10\alpha ^2-9\alpha +4=0$.) This causes a minimal ``spill off" of $G^4(W)$ outside $T$. See Figure \ref{fig:trap059-065_3i4}. We also see there that $G^5(W)\subset T$.

\begin{figure}[h] 
  \centering
  \includegraphics[bb=0 -1 662 339,width=3.92in,height=2.01in,keepaspectratio]{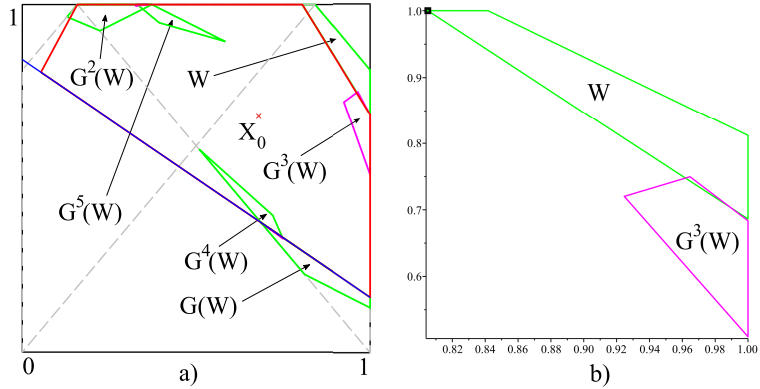}
  \caption{$\alpha =0.594$ (case i)) a)  Region $W$ and its images in green except for $G^3(W)$ in magenta, $G^5(W)\subset T$.
b) Enlargement of the intersection of $W$ and $G^3(W)$ which causes $G^4(W)\not\subset T$.}
  \label{fig:trap059-065_3i4}
\end{figure}

ii) $(\sim 0.5970091680,\sim 0.6513878188)$ 

For $\alpha\in (\sim 0.5970091680,(\sqrt{13}-1)/4=\sim 0.6513878188)$ the set $W$ and $G^3(W)$ no longer intersect and
 $G^4(W)\subset T$. See Figure \ref{fig:trap059-065_1i2} b). Value $\alpha=(\sqrt{13}-1)/4$ is the point where $W$ stops to be a quadrangle and starts to be just a triangle.

\begin{figure}[h] 
  \centering
  \includegraphics[bb=0 -1 661 339,width=3.92in,height=2.01in,keepaspectratio]{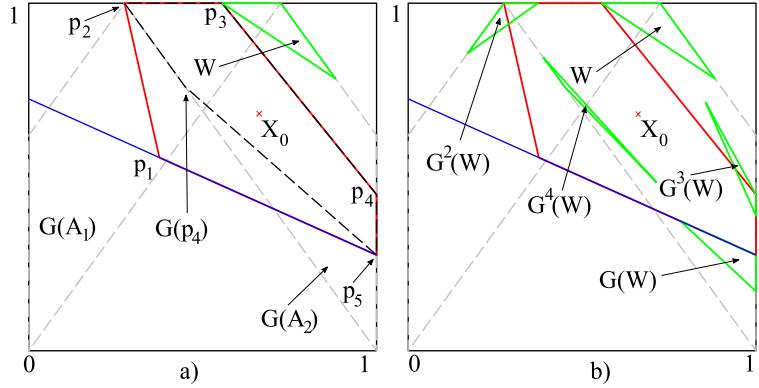}
  \caption{$\alpha =0.69$ (case iii)) a) Trapping region $T$ (red) and its image $G(T)$ (dashed black). b) Region $W$ and its images, $G^4(W)\subset T$.}
  \label{fig:trap065-0728aib}
\end{figure}

iii)$(\sim 0.6513878188, \sim 0.7287135539)$ 

For $\alpha$ between $\sim 0.6513878188$ and $1/4+\sqrt(33)/12=\sim 0.7287135539$,
the region $W$ is a triangle and $G^4(W)\subset T$. See Figure \ref{fig:trap065-0728aib}. Part a) shows the trapping region $T$ (red) and its image $G(T)$ (dashed black). Part b) shows region $W$ and its images, $G^4(W)\subset T$.
For $\alpha$ approaching $0.7287135539$ the top vertex of $G^4(W)$ approaches boundary of $T$ but stays in $T$ as it is the image of
the lowest vertex of $G^3(W)$ which is already in $T$. For $\alpha$ above $\sim 0.7287135539$ the image $G(p_4)$ goes outside
the line $L(p_1,p_2)$ and $T$ is no longer a trapping region.
\end{proof}


For the next interval of parameter $\alpha$ we have to make a ``micro" adjustment of $T$ adding to its construction
two more vertices $G(p_4)$ and $G^2(p_4)$.

\begin{proposition}\label{pr0728-0736} For $\sim 0.7287135539<\alpha\le \sim 0.7360241475$, the fixed point $X_0=(2/3,2/3)$ attracts all points except $(0,0)$. $\sim 0.7360241475$ is the root of $4\alpha^4-8\alpha^3+14\alpha^2-13\alpha+4=0$.
Above this value of $\alpha$ sets $W$ and $G^2(W)$ intersect.
\end{proposition}

\begin{figure}[h] 
  \centering
  \includegraphics[bb=0 -1 661 339,width=3.92in,height=2.01in,keepaspectratio]{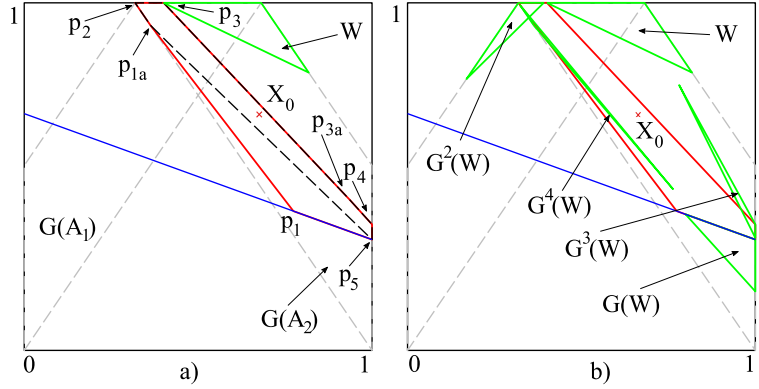}
  \caption{$\alpha=0.734$ a)  the trapping region $T$ (red) and its image $G(T)$ (dashed black). b) shows $W$ and its images with $G^4(W)\subset T$.}
  \label{fig:Trap07287-0736aib}
\end{figure}

\begin{proof} Again, we construct a trapping region $T$ and show that fourth  image of $W=G_2^{-1}(A_1)\cap G(A_2)$
falls into $T$. The construction of $T$ is a micro adjustment of the construction from Proposition \ref{pr0593-0728},
it is almost not visible on pictures. We add two more vertices , $p_{1a}=G(p_4)$ and $p_{3a}=G^2(p_4)$,
to the the construction and $T$ becomes a heptagon (seven angles figure). Since $G(p_{3a})$ is inside such constructed  $T$,
and $T$ is convex, we have $G_2(T)\subset(T)$. See Figure \ref{fig:Trap07287-0736aib}. Part a) shows the trapping region $T$ (red) and its image $G(T)$ (dashed black). The green triangle is the region $W$.
$G(p_{3a})$ stays inside  $T$ up to $\alpha=\sim 0.7464180853$ but earlier another problem arises. At $\alpha= \sim 0.7360241475$
the image $G^2(W)$ starts intersecting with $W$ and this needs another approach.

Figure \ref{fig:Trap07287-0736aib} b) shows $W$ and its images with $G^4(W)\subset T$. Two upper vertices of $G^4(W)$ are on the boundary of $T$ since the corresponding vertices of $G^3(W)$ are already
on the boundary of $T$. This is better visible on the  Figure \ref{fig:Trap07287-0736Enlarged} b) presenting $T$, $G^3(W)$ and $G^4(W)$.
Figure \ref{fig:Trap07287-0736Enlarged} a) shows the old trapping region of Proposition \ref {pr0593-0728} and the points $G(p_4)$, $G^2(p_4)$ both outside this region as well as the point $G^3(p_4)=G(p_{3a})$
well inside $T$.

\begin{figure}[h] 
  \centering
  \includegraphics[bb=0 -1 645 332,width=3.92in,height=2.02in,keepaspectratio]{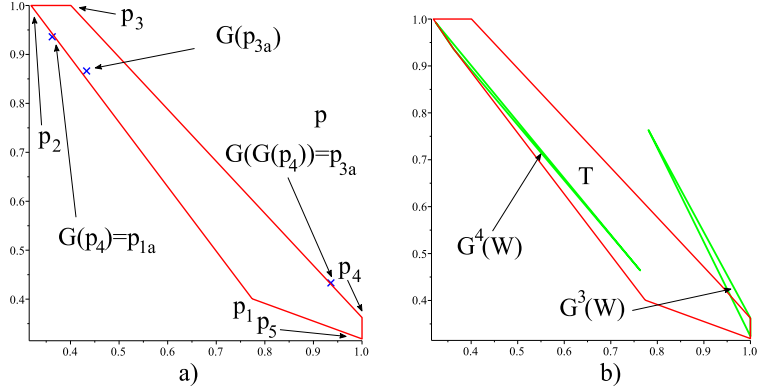}
  \caption{$\alpha=0.734$ a)the old trapping region of Proposition \ref {pr0593-0728} and the points $G(p_4)$, $G^2(p_4)$, $G^3(p_4)$. b) enlarged $T$, $G^3(W)$ and $G^4(W)$.}
  \label{fig:Trap07287-0736Enlarged}
\end{figure}

Now, we will consider the last subinterval of $\alpha$'s for which $X_0$ is an almost global attractor.


\begin{proposition}\label{pr0736-075} For $\sim 0.7360241475<\alpha< 3/4$, the fixed point $X_0=(2/3,2/3)$ attracts all points except $(0,0)$. 
\end{proposition}

\begin{proof} For $\sqrt{3}-1<\alpha<3/4$, $\sqrt{3}-1=\sim 0.732050808$, the eigenvalues of $DG_2$ are real and both 
between $-1$ and $-1/2$. They are $\lambda_{1,2}=-\alpha\pm \sqrt{\alpha^2+2\alpha-2}$.
The corresponding eigenvectors are $v_{1,2}=[(-\alpha\pm \sqrt{\alpha^2+2\alpha-2})^{-1},1]$.

\begin{figure}[h] 
  \centering
  \includegraphics[bb=0 -1 661 340,width=3.92in,height=2.02in,keepaspectratio]{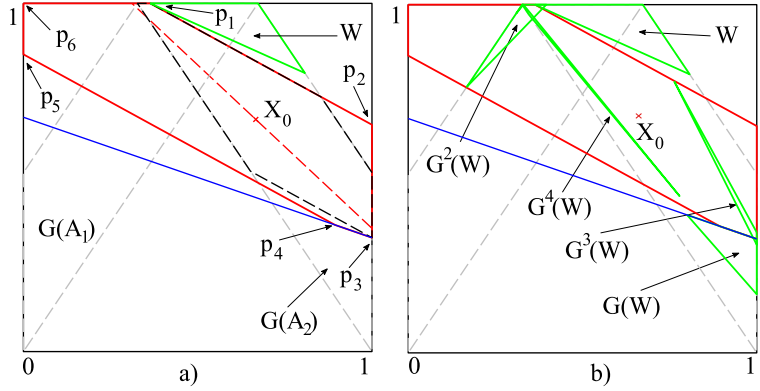}
  \caption{$\alpha=0743$ a) Trapping region $T$ (red) and its image $G(T)$ (dashed black). The dashed red line is an eigenline going through $X_0$.
b) Region $W$  and its images (green), $G^4(W)\subset T$.}
  \label{fig:Trap0736-075aib}
\end{figure}

Since $\alpha$'s up to $\sim 0.7360241475$ were already considered, we will study only the interval $(\sim 0.7360241475,3/4)$.
The trapping region will be constructed using the the vector $v_1$, see Figure \ref{fig:Trap0736-075aib} a). Let $p_1$ be the left upper vertex
of $W=G_2^{-1}(A_1)\cap G_2(A_2)$ and $p_4=G_2^{-1}(p_1)$ its preimage on the partition line. $T$ is the part of $A_2$ 
between the lines $L_1$, $L_4$ going through points $p_1$ and $p_4$, respectively,  and parallel to the vector $v_1$. Thus, $T$ is a hexagon
with vertices $p_1$, $p_2=L_1\cap \{x=1\}$, $p_3= {\rm partition line}\cap  \{x=1\}$, $p_4$, $p_5=L_4\cap \{x=0\}$ and
$p_6=(0,1)$.  $T$ is a trapping region, $G(T)\subset T$, by construction since its sides are the eigenlines and eigenvalues have absolute values less than one. Figure \ref{fig:Trap0736-075aib} a) shows 
the trapping region $T$ (red) and its image $G(T)$ (dashed black). The dashed red line is an eigenline (parallel to $v_2$) going through $X_0$.

\begin{figure}[h] 
  \centering
  \includegraphics[bb=0 -1 651 329,width=3.92in,height=1.98in,keepaspectratio]{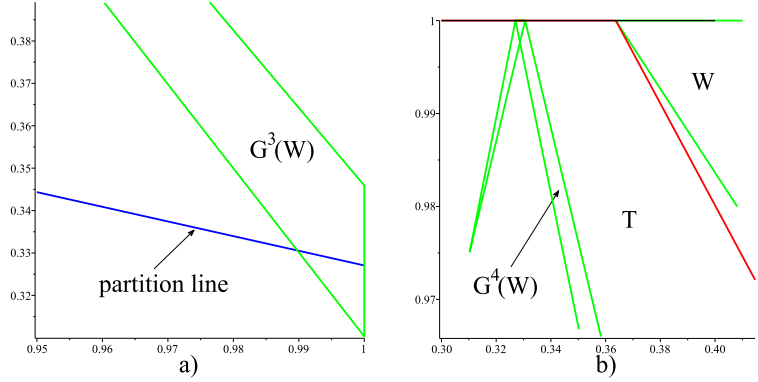}
  \caption{$\alpha=0743$ a) Lower part of $G^3(W)$ and b)
upper part of $G^4(W)$.}
  \label{fig:Trap0736-075Enlargedaib}
\end{figure}

In Figure \ref{fig:Trap0736-075aib} b) we see region $W$ and its images (green). We see that $G^4(W)\subset T$. It can be proven that the lowest vertex of $G^4(W)$ touches the line $L(p_4,p_5)$ first time for $\alpha=3/4$. Lower part of $G^3(W)$ and 
upper part of $G^4(W)$ are shown more precisely in Figure \ref{fig:Trap0736-075Enlargedaib} a) and b), respectively.
Since $G^3(W)$ crosses the partition line, its image $G^4(W)$ is ``broken".
\end{proof}

Propositions \ref{pr05-0593}, \ref{pr0593-0728}, \ref{pr0728-0736} and \ref {pr0736-075} together prove the following:

\begin{theorem} \label{Th05-075} For $1/2<\alpha< 3/4$, the fixed point $X_0=(2/3,2/3)$ attracts all points except $(0,0)$, so it is an almost global attractor. 
\end{theorem}

\end{proof}


\end{document}